\documentclass[graybox]{svmult}

\usepackage{type1cm}        
\usepackage{makeidx}         
\usepackage{graphicx}        
\usepackage{multicol}        
\usepackage[bottom]{footmisc}

\usepackage{newtxtext}       %
\usepackage[varvw]{newtxmath}       

\makeindex             

\def\nmo{{n\hn-\nh1}}
\def\hyp{\hskip.5pt\vbox
{\hbox{\vrule width2.5ptheight0.5ptdepth0pt}\vskip2pt}\hskip.5pt}

\def\hs{\hskip.7pt}
\def\hh{\hskip.4pt}
\def\nh{\hskip-.7pt}

 1

\def\bg{\varPi}
\def\rd{\delta}

\def\lp{\alpha}
\def\stb{\varSigma}
\def\sg{\mathcal{G}}
\def\tvs{\mathcal{V}}
\def\tws{\mathcal{W}}
\def\w{^{\phantom i}}
\def\nnh{\hskip-1.5pt}
\def\hn{\hskip-.4pt}

\def\bbR{\mathrm{I\!R}}
\def\rn{\bbR\nh^n}
\def\bbQ{{\mathchoice{\setbox0=\hbox{$\displaystyle\rm Q$}\hbox{\raise
0.15\ht0\hbox to0pt{\kern0.4\wd0\vrule height0.8\ht0\hss}\box0}}
{\setbox0=\hbox{$\textstyle\rm Q$}\hbox{\raise
0.15\ht0\hbox to0pt{\kern0.4\wd0\vrule height0.8\ht0\hss}\box0}}
{\setbox0=\hbox{$\scriptstyle\rm Q$}\hbox{\raise
0.15\ht0\hbox to0pt{\kern0.4\wd0\vrule height0.7\ht0\hss}\box0}}
{\setbox0=\hbox{$\scriptscriptstyle\rm Q$}\hbox{\raise
0.15\ht0\hbox to0pt{\kern0.4\wd0\vrule height0.7\ht0\hss}\box0}}}}
\newcommand{\bbC}{{\mathchoice {\setbox0=\hbox{$\displaystyle\mathrm{C}$}
\hbox{\hbox to0pt{\kern0.4\wd0\vrule height0.9\ht0\hss}\box0}} 
{\setbox0=\hbox{$\textstyle\mathrm{C}$}\hbox{\hbox 
to0pt{\kern0.4\wd0\vrule height0.9\ht0\hss}\box0}} 
{\setbox0=\hbox{$\scriptstyle\mathrm{C}$}\hbox{\hbox 
to0pt{\kern0.4\wd0\vrule height0.9\ht0\hss}\box0}} 
{\setbox0=\hbox{$\scriptscriptstyle\mathrm{C}$}\hbox{\hbox 
to0pt{\kern0.4\wd0\vrule height0.9\ht0\hss}\box0}}}}

\def\bbZ{\mathsf{Z\hskip-3.5ptZ}}
\def\dimr{\dim_{\hskip.3pt\bbR\hskip-1.7pt}\w}
\def\dimq{\dim_{\hskip.8pt\bbQ\hskip-1.7pt}\w}
\def\dimz{\dim_{\hskip.1pt\mathsf{Z\hskip-2.8ptZ}\hskip-1.7pt}\w}
\def\spr{\mathrm{span}_{\hskip-.2pt\bbR\hskip-1.7pt}\w}

\def\w{^{\phantom i}}
\def\af{af\hn\-fine}
\def\Af{Af\nh\-fine}

\begin{document}

\title*{Compact flat manifolds and reducibility}
\author{Andrzej Derdzinski\orcidID{0000-0001-9015-0886} and\\
Paolo Piccione\orcidID{0000-0002-8929-5938}}
\institute{Andrzej Derdzinski \at Department of Mathematics,
The Ohio State University, 231 W. 18th Avenue, Columbus, OH 43210, USA,
\email{andrzej@math.ohio-state.edu}
\and Paolo Piccione \at Departamento de Matem\'atica, 
Instituto de Matem\'atica e Estat\'\i stica, Universidade de S\~ao Paulo, 
Rua do Mat\~ao 1010, CEP 05508-900, S\~ao Paulo, SP, Brazil, 
\email{piccione@ime.usp.br}}
%
%
\maketitle

\abstract*{Hiss and Szcze\-pa\'n\-ski proved in 1991 that the holonomy group
of any compact 
flat Riemannian manifold, of dimension at least two, acts reducibly on the 
rational span of the Euclidean lattice associated with the manifold via the 
first Bieberbach theorem. Geometrically, their result states that such a 
manifold must admit a nonzero proper parallel distribution with compact 
leaves. We study algebraic and geometric properties of the sublattice-spanned 
holonomy-invariant rational vector subspaces that exist due to the above 
theorem, and of the resulting compact-leaf foliations of compact flat 
manifolds. The class consisting of the former subspaces, in addition to being 
closed under spans and intersections, also turns out to admit (usually 
nonorthogonal) complements. As for the latter foliations, we provide 
descriptions, first -- and foremost -- of the intrinsic geometry of their 
generic leaves in terms of that of the original flat manifold and, secondly -- 
as an essentially obvious afterthought -- of the leaf-space orbifold. The 
general conclusions are then illustrated by examples in the form of 
generalized Klein bottles.}

\abstract{Hiss and Szcze\-pa\'n\-ski proved in 1991 that the holonomy group of
any compact 
flat Riemannian manifold, of dimension at least two, acts reducibly on the 
rational span of the Euclidean lattice associated with the manifold via the 
first Bieberbach theorem. Geometrically, their result states that such a 
manifold must admit a nonzero proper parallel distribution with compact 
leaves. We study algebraic and geometric properties of the sublattice-spanned 
holonomy-invariant rational vector subspaces that exist due to the above 
theorem, and of the resulting compact-leaf foliations of compact flat 
manifolds. The class consisting of the former subspaces, in addition to being 
closed under spans and intersections, also turns out to admit (usually 
nonorthogonal) complements. As for the latter foliations, we provide 
descriptions, first -- and foremost -- of the intrinsic geometry of their 
generic leaves in terms of that of the original flat manifold and, secondly -- 
as an essentially obvious afterthought -- of the leaf-space orbifold. The 
general conclusions are then illustrated by examples in the form of 
generalized Klein bottles.}

\section{Introduction}\label{in}
As shown by Hiss and Szcze\-pa\'n\-ski 
\cite[the corollary in Sect.\ 1]{hiss-szczepanski}, on any compact flat 
Riemannian manifold $\,\mathcal{M}\hs$ with $\,\dim\mathcal{M}=n\ge2\,$ there 
exists a parallel distribution $\,D\hs$ of dimension $\,k$, where $\,0<k<n$, 
such that the leaves of $\,D\hs$ are all compact. In the Appendix we 
reproduce the original algebraic phrasing of their result and mention a
stronger version of it, established more recently by Lu\-tow\-ski
\cite{lutowski}, which implies that -- unless $\,\mathcal{M}\hs$ is a flat
torus -- there exist at least two distributions $\,D\hs$ with the above
properties, having non\-e\-quiv\-a\-lent irreducible holonomy representations.

The present paper deals with geometric consequences of Hiss and
Szcze\-pa\'n\-ski's theorem. We do not pursue analogous ramifications of
Lu\-tow\-ski's generalization.

Our main results are Theorems~\ref{restr}, \ref{twopr}, \ref{gnric},
\ref{intrs}, \ref{orien} and Corollary~\ref{clsed}.

Theorem~\ref{restr} describes geometries of the individual leaves 
$\,\mathcal{M}\hn'$ of a distribution $\,D$ on $\,\mathcal{M}\hs$ having 
the properties mentioned above, in terms of the short exact sequence 
$\,L\to\hs\bg\to H\hs$ formed by the lattice $\,L$, holonomy group $\,H\hs$ 
and Bie\-ber\-bach group $\,\bg$ associated with $\,\mathcal{M}$, and 
its analog $\,L\nh'\nh\to\hs\bg\hn'\nh\to H\nh'$ for $\,\mathcal{M}\hn'\nh$. 
Specifically, according to our Theorem~\ref{restr}(ii), $\,\bg\hn'$ 
(or, $\,L\nh'$) may be treated as a sub\-group of $\,\bg\,$ (or, of a 
certain Euclidean vector space 
$\,\tvs$), and $\,H\nh'$ as a homo\-mor\-phic image of a sub\-group of 
$\,H\nnh$.

In Theorem~\ref{twopr} we establish the existence, on 
every compact flat Riemannian manifold $\,\mathcal{M}\hs$ of dimension 
$\,n\ge2$, of {\it two\/} proper parallel distributions $\,D\hs$ and 
$\,\hat D$ with compact leaves, complementary to each other 
in the sense that $\,T\nh\mathcal{M}\nh=D\oplus\hat D\nh$.

Corollary~\ref{clsed} states that, in any manifold $\,\mathcal{M}\,$ as above,
the class of parallel distributions with compact leaves is closed under spans
and intersections.

Theorem~\ref{gnric} addresses the particularly simple form of the sequence 
$\,L\nh'\nh\to\hs\bg\hn'\nh\to H\nh'$ described by Theorem~\ref{restr}(ii),
arising in the case of leaves 
$\,\mathcal{M}\hn'$ which we call {\it generic}. 
The union of all generic leaves is an open dense subset of $\,\mathcal{M}$,
they all have the same 
triple $\,L\nh'\nnh,\bg\hn'\nnh,H\nh'\nnh$, and are mutually isometric. When 
all leaves of $\,D\,$ happen to be generic, they form a locally trivial bundle 
with compact flat manifolds serving both as the base and the fibre (the 
{\it fibration case}).

Theorems~\ref{intrs} and \ref{orien} describe the intersection numbers
of the leaves of the two mutually complementary foliations resulting
from Theorem~\ref{twopr}.

Aside from the holonomy group $\,H\nh'$ of each individual leaf 
$\,\mathcal{M}\hn'$ of $\,D\nh$, forming a part of its intrinsic 
(sub\-man\-i\-fold) geometry, $\,\mathcal{M}\hn'$ also gives rise to two 
``extrinsic'' holonomy groups, one arising since $\,\mathcal{M}\hn'$ is a leaf 
of the foliation $\,F\hskip-3.8pt_{\mathcal{M}}\w$ of $\,\mathcal{M}\,$ 
tangent to $\,D\nh$, the other coming from the normal connection of 
$\,\mathcal{M}\hn'\nnh$. Due to the flatness of the normal connection, the two 
extrinsic holonomy groups coincide, and are trivial for all generic leaves. In 
Sect.~\ref{ls} we briefly discuss the leaf space 
$\,\mathcal{M}/\nh F\hskip-3.8pt_{\mathcal{M}}\w$, pointing out that (not 
surprisingly!) $\,\mathcal{M}/\nh F\hskip-3.8pt_{\mathcal{M}}\w$ is a flat  
compact or\-bi\-fold which, in the fibration case mentioned above, constitutes
the base manifold of the bundle.

Our results are illustrated by examples (generalized Klein bottles, 
Sect.~\ref{gk}), where both the fibration and non-fibration cases occur, 
depending on the choice of $\,D\nh$. 

In Sect.~\ref{rh} we discuss certain analogs of the correspondence between
Bie\-ber\-bach groups and compact flat manifolds: one provided by
al\-most-Bie\-ber\-bach groups and in\-fra-nil\-man\-i\-folds, the other by 
the group $\,\mathrm{Spin}\hs(m,1)\,$ acting on the or\-tho\-nor\-mal-frame
bundle of the hyperbolic $\,m$-space, leading to quotient manifolds that 
include some compact locally symmetric 
pseu\-\hbox{do\hskip1pt-}\hskip-.7ptRiem\-ann\-i\-an Ein\-stein manifolds.

An earlier version \cite{derdzinski-piccione} of this paper is cited by  
\cite{bettiol-derdzinski-mossa-piccione}, and therefore still available on the 
arXiv. The presentation in \cite{derdzinski-piccione} was -- as we eventually
realized -- rather far from read\-er-friend\-ly, which prompted us to
thoroughly rewrite the whole text.

\section{Preliminaries}\label{pr}
Manifolds, mappings (except in Lemma~\ref{opdns}) and tensor fields, such as
bundle and covering projections, sub\-man\-i\-fold inclusions, and Riemannian
metrics, are by definition of class $\,C^\infty\nnh$. Sub\-man\-i\-folds need
not carry the subset topology, and a manifold may be disconnected (although,
being required to satisfy the second countability axiom, it must have at most
countably many connected components). Con\-nect\-ed\-ness/com\-pact\-ness of a 
sub\-man\-i\-fold always refers to its own topology, and implies the same for 
its underlying set within the ambient manifold. Thus, a compact 
sub\-man\-i\-fold is always endowed with the subset topology.

By a {\it distribution\/} on a manifold $\,\mathcal{N}\hs$ we mean, as 
usual, a (smooth) vector sub\-bundle $\,D\,$ of the tangent bundle 
$\,T\mathcal{N}\nh$. An {\it integral manifold\/} of $\,D\,$ is any 
sub\-man\-i\-fold $\,\mathcal{L}\,$ of $\,\mathcal{N}\hs$ with 
$\,T\hskip-2.8pt_x\w\nh\mathcal{L}=D\nnh_x\w$ whenever $\,x\in\mathcal{L}$.
The maximal connected integral manifolds of $\,D\,$ will also be referred to
as the {\it leaves\/} of $\,D\nh$. If $\,D\,$ is integrable, 
its leaves form the foliation associated with $\,D\nh$. We call $\,D\,$ 
{\it pro\-ject\-a\-ble\/} under a mapping 
$\,\psi:\mathcal{N}\nh\to\hat{\mathcal{N}}$ onto a distribution 
$\,\hat D\,$ on the target manifold $\,\hat{\mathcal{N}}\hs$ if 
$\,d\psi\nh\nnh_x\w(D\nnh_x\w)=\hat D\nnh_{\psi(x)}$ for all 
$\,x\in\mathcal{N}\nnh$.
\begin{remark}\label{covpr}The following facts are well known: (c) is the 
compact case of Ehres\-mann's fibration theorem 
\cite[Corollary 8.5.13]{dundas}; (b) follows from (c). For (a), see 
\cite[pp.\ 43--44 and 61--62]{kobayashi-nomizu} -- note that finiteness
trivially implies proper discontinuity.
\begin{enumerate}
  \def\theenumi{{\rm\alph{enumi}}}
\item Free dif\-feo\-mor\-phic actions of finite groups on manifolds are 
properly discontinuous, leading to covering projections onto the 
resulting quotient manifolds.
\item Any lo\-cal\-ly-dif\-feo\-mor\-phic mapping from a compact manifold into 
a connected manifold is a (surjective) finite covering projection.
\item More generally, the phrases `lo\-cal\-ly-dif\-feo\-mor\-phic mapping' 
and `finite covering projection' in (b) may be replaced with {\it 
submersion\/} and {\it fibration}. 
\end{enumerate}
\end{remark}
\begin{lemma}\label{cptlv}
Given manifolds\/ $\,\hat{\mathcal{M}}\,$ and\/ $\,\mathcal{M}\,$ with
distributions\/ $\,\hat D\,$ and\/ $\,D$, let\/ $\,\hat D\,$ be
pro\-ject\-a\-ble onto\/ $\,D\,$ under a locally dif\-feo\-mor\-phic
surjective mapping\/ $\,\psi:\hat{\mathcal{M}}\to\mathcal{M}$.
\begin{enumerate}
  \def\theenumi{{\rm\roman{enumi}}}
\item The\/ $\,\psi$-im\-age of any leaf of\/ $\,\hat D\,$ is a connected 
integral manifold of\/ $\,D\nh$.
\item Integrability of\/ $\,\hat D\,$ implies that of\/ $\,D\nh$.
\item For any compact leaf\/ $\,\mathcal{L}\,$ of $\,\hat D\nh$, the image\/ 
$\,\mathcal{L}\nh'\nh=\psi(\mathcal{L})\,$ is a compact leaf of $\,D\nh$, and 
the restriction\/ $\,\psi:\mathcal{L}\nh\to\mathcal{L}\nh'$ constitutes 
a covering projection.
\item If the leaves of\/ $\,\hat D\,$ are all compact, so are those of\/ 
$\,D\nh$.
\end{enumerate}
\end{lemma}
\begin{proof}Assertion (i) is immediate from the definitions of a leaf and 
projectability, while (i) yields (ii) as integrability amounts to the 
existence of an integral manifold through every point. Remark~\ref{covpr}(b) 
and (i) give (iii). Now (iv) follows.
\end{proof}
\begin{lemma}\label{opdns}
Suppose that\/ $\,F\,$ is a mapping from a 
manifold\/ $\,\mathcal{W}$ into any set. If \,for every\/ 
$\,x\in\mathcal{W}\,$ there exists a dif\-feo\-mor\-phic identification of a 
neighborhood\/ $\,\mathcal{B}\nnh_x\w$ of\/ $\,x\,$ in\/ $\,\mathcal{W}\,$ 
with a unit open Euclidean ball centered at\/ $\,0\,$ under which\/ $\,x\,$ 
corresponds to\/ $\,0\,$ and\/ $\,F\hs$ becomes constant on each open 
straight-line interval of length\/ $\,1\,$ in the open Euclidean ball having\/ 
$\,0\,$ as an endpoint, then\/ $\,F\,$ is locally constant on some open dense 
subset of\/ $\,\mathcal{W}\nh$.
\end{lemma}
\begin{proof}We use induction on $\,n=\dim\mathcal{W}\nh$. The case $\,n=1\,$ 
being trivial, let us suppose that the assertion holds in dimension $\,\nmo\,$ 
and consider a mapping $\,F\hs$ from an $\,n$-di\-men\-sion\-al manifold 
$\,\mathcal{W}\nh$, satisfying our hypothesis, along with an embedded open 
Euclidean ball 
$\,\mathcal{B}\nnh_x\w\nh\subseteq\mathcal{M}\,$ ``centered'' at a fixed point 
$\,x$, as in the statement of the lemma. The constancy of $\,F\nh$ along 
the fibres of the normalization projection 
$\,\mu:\mathcal{B}\nnh_x\w\nh\smallsetminus\{x\}\to\mathcal{S}$ onto the unit 
$\,(n\nh-\nh1)$-sphere $\,\mathcal{S}\,$ gives rise to a mapping 
$\,G\,$ with the domain $\,\mathcal{S}\,$ and $\,F=G\circ\mu$. Let us now fix
$\,s\in\mathcal{S}$, any $\,y\in\mathcal{B}\nnh_x\w\nh\smallsetminus\{x\}\,$
with $\,\mu(y)=s$, and an embedded open 
Euclidean ball $\,\mathcal{B}\nnh_y\w$ ``centered'' at $\,y$, such that 
$\,F\hs$ is constant on each radial open interval in $\,\mathcal{B}\nnh_y\w$. 
The obvious sub\-mer\-sion property of $\,\mu$ allows us to pass from 
$\,\mathcal{B}\nnh_y\w$ to a smaller concentric ball and then choose a 
co\-di\-men\-sion-one open Euclidean ball $\,\mathcal{B}'_{\!y}$ arising as a 
union of radial intervals within this smaller version of 
$\,\mathcal{B}\nnh_y\w$, for which $\,\mu:\mathcal{B}'_{\!y}\to\mathcal{S}\,$ 
is an embedding. The assumption of the lemma thus holds when 
$\,\mathcal{W}\hs$ and $\,F\hs$ are replaced by $\,\mathcal{S}\,$ and $\,G$, 
leading to the local constancy of $\,G\,$ (and $\,F$) on a dense open set in 
$\,\mathcal{S}$ (and, respectively, in 
$\,\mathcal{B}\nnh_x\w\nh\smallsetminus\{x\}$). Since the union of the latter 
sets over all $\,x\,$ is obviously dense in $\,\mathcal{W}\nh$, our claim 
follows.
\end{proof}
We have the following well-known consequence of the inverse mapping theorem
combined with the Gauss lemma for sub\-man\-i\-folds.
\begin{lemma}\label{nrexp}Given a compact 
sub\-man\-i\-fold\/ $\,\mathcal{M}\hn'$ of a Riemannian manifold\/ 
$\,\mathcal{M}$, every sufficiently small\/ $\,\rd\in(0,\infty)\,$ has the
following properties.
\begin{enumerate}
  \def\theenumi{{\rm\alph{enumi}}}
\item The normal exponential mapping restricted to the 
radius\/ $\,\rd\,$ open-disk sub\-bun\-dle\/ $\,\mathcal{N}\hskip-3pt_\rd\w$
of the normal bundle of\/ $\,\mathcal{M}\hn'$ constitutes a 
dif\-feo\-mor\-phi
sm\/ $\,\mathrm{Exp}^\perp\nnh:\mathcal{N}\hskip-3pt_\rd\w
\to\nh\mathcal{M}'\hskip-4.1pt_\rd\w$ 
onto the open sub\-man\-i\-fold\/ $\,\mathcal{M}'\hskip-4.1pt_\rd\w$ of\/
$\,\mathcal{M}\,$ equal to the pre\-im\-age of\/ $\,[\hh0,\rd)\,$ under the
function\/ $\,\mathrm{dist}\hh(\mathcal{M}\hn'\nh,\,\cdot\,)\,$ of metric
distance from\/ $\,\mathcal{M}\hn'\nnh$.
\item Each\/ $\,x\in\mathcal{M}'\hskip-4.1pt_\rd\w$ has a unique point\/  
$\,y\in\mathcal{M}\hn'$ nearest to\/ $\,x$, which is simultaneously the unique 
point\/ $\,y\,$ of\/ $\,\mathcal{M}\hn'$ joined to\/ $\,x\,$ by a geodesic
in\/ $\,\mathcal{M}'\hskip-4.1pt_\rd\w$ normal to\/ $\,\mathcal{M}\hn'$ at
$\,y$, and the resulting assignment\/
$\,\mathcal{M}'\hskip-4.1pt_\rd\w\ni x\mapsto y\in\mathcal{M}\hn'$ coincides
with the composition of the inverse dif\-feo\-mor\-phism of\/
$\,\mathrm{Exp}^\perp\nnh:\mathcal{N}\hskip-3pt_\rd\w
\to\nh\mathcal{M}'\hskip-4.1pt_\rd\w\,$ 
followed by the nor\-mal-bun\-dle projection\/ 
$\,\mathcal{N}\hskip-3pt_\rd\w\to\nh\mathcal{M}\hn'\nnh$.
\item The\/ $\,\mathrm{Exp}^\perp$ images of length\/ $\,\rd\,$ radial line 
segments emanating from the zero vectors in the fibres of\/ 
$\,\mathcal{N}\hskip-3pt_\rd\w$ coincide with the length\/ $\,\rd\,$
minimizing geodesic segments in\/ $\,\mathcal{M}'\hskip-4.1pt_\rd\w$ emanating from\/ 
$\,\mathcal{M}\hn'$ and normal to\/ $\,\mathcal{M}\hn'\nnh$. They are also
normal to all the levels of\/ 
$\,\mathrm{dist}\hh(\mathcal{M}\hn'\nh,\,\cdot\,)\,$ in\/
$\,\mathcal{M}'\hskip-4.1pt_\rd\w$, and realize the minimum distance between any two
such levels within\/ $\,\mathcal{M}'\hskip-4.1pt_\rd\w$.
\end{enumerate}
\end{lemma}
\begin{lemma}\label{baire}
In a complete metric space, any countable union of closed sets with empty 
interiors has an empty interior.
\end{lemma}
\begin{proof}This is Baire's theorem \cite[p.\ 187]{edwards} stating, 
equivalently, that the intersection of countably many dense open subsets 
is dense.
\end{proof}

\section{Free Abel\-i\-an groups}\label{fa}
The following well-known facts, cf.\ \cite[p.\ 2]{arnold}, are gathered here
for easy reference.

For a finitely generated Abel\-i\-an group $\,G$, being tor\-sion-free amounts 
to being free, in the sense of having a $\,\bbZ\nh$-ba\-sis, by which one 
means an ordered $\,n$-tuple $\,e\nnh_1\w,\dots,e\nnh_n\w$ of elements of $\,G\,$ such 
that every $\,x\in G\,$ can be uniquely expressed as an integer combination of 
$\,e\nnh_1\w,\dots,e\nnh_n\w$. The integer $\,n\ge0$, also denoted by $\,\dimz G$, 
is an algebraic invariant of $\,G$, called its {\it rank}, or {\it Bet\-ti
number}, or $\,\bbZ\nh${\it-di\-men\-sion}.

Any finitely generated Abel\-i\-an group $\,G\,$ is isomorphic to the direct 
sum of its (necessarily finite) torsion sub\-group $\,S\,$ and the free group 
$\,G/S$. We then set $\,\dimz G=\dimz\hskip2pt[G/S]$. A sub\-group $\,G'$ (or, 
a homo\-mor\-phic image $\,G'$) of such $\,G$, in addition to being again 
finitely generated and Abel\-i\-an, also satisfies the inequality 
$\,\dimz G'\le\dimz G$, strict unless $\,G/G'$ is finite (or, 
respectively, the homo\-mor\-phism in question has a finite kernel).
\begin{lemma}\label{dirsm}
A sub\-group\/ $\,G'$ of a finitely generated free Abel\-i\-an group\/
$\,G\,$ 
constitutes a direct summand of\/ $\,G\,$ if and only if the quotient group\/ 
$\,G/G'$ is tor\-sion-free.
\end{lemma}
In fact, more generally, given a surjective homo\-mor\-phism 
$\,\chi:P\to P'$ between Abel\-i\-an groups $\,P,P'$ and elements 
$\,x\hskip-2.3pt_j\w,y\hskip-2pt_a\w\in P\hs$ (with $\,j,a\,$ ranging over
finite sets), such that $\,x\hskip-2.3pt_j\w$ and $\,\chi(y\hskip-2pt_a\w)\,$
happen to form $\,\bbZ\nh$-ba\-ses of 
$\,\mathrm{Ker}\hskip2.7pt\chi\,$ and, respectively, of $\,P'\nnh$, the system 
consisting of all $\,x\hskip-2.3pt_j\w$ and $\,y\hskip-2pt_a\w\,$ is a
$\,\bbZ\nh$-ba\-sis of 
$\,P\,$ (and so $\,P\,$ must be free). This is clear as every element of
$\,P'$ (or, of $\,P$) then can be uniquely expressed as an integer combination
of $\,\chi(y\hskip-2pt_a\w)\,$ (or, consequently, of
$\,x\hskip-2.3pt_j\w$ and $\,y\hskip-2pt_a\w$).
\begin{lemma}\label{surjc}
For each finitely generated sub\-group\/ $\,G\,$ of the additive group of a 
fi\-nite-di\-men\-sion\-al real vector space\/ $\,\tvs\nnh$, the 
intersection\/ $\,G\cap\nh\tvs\hn'$ with any vector sub\-space\/ 
$\,\tvs\hn'\subseteq\tvs\hs$ forms a di\-rect-sum\-mand sub\-group of\/ 
$\,G$. Furthermore, the class of di\-rect-sum\-mand sub\-groups of\/ $\,G\,$
is closed under intersections, finite or not.
\end{lemma}
Both claims are obvious from Lemma~\ref{dirsm}. Next, we have a 
straightforward exercise:
\begin{lemma}\label{prdsg}
If normal sub\-groups\/ $\,G'\nh,G''$ of a group\/ $\,G\,$ intersect trivially 
and every\/ $\,\gamma'\in G'$ commutes with every\/ $\,\gamma''\in G''\nnh$, 
then\/ 
$\,G'G''\nh=\{\gamma'\gamma'':(\gamma'\nnh,\gamma'')\in G'\nnh\times G''\hn\}$ 
is a normal sub\-group\/ of\/ $\,G$, and the assignment\/ 
$\,(\gamma'\nnh,\gamma'')\mapsto\gamma'\gamma''$ defines an iso\-mor\-phism\/ 
$\,G'\nnh\times G''\nh\to G'G''\nnh$.
\end{lemma}

\section{Lattices and vector sub\-spaces}\label{lv}
Throughout this section $\,\tvs\hs$ denotes a fixed fi\-nite-di\-men\-sion\-al 
real vector space, and
\begin{equation}\label{cpl}
\mathrm{sub\-spaces\ }\,\tvs\hn'\nnh,\tvs\hn''\subseteq\tvs\hs\mathrm{\ with\
}\,\tvs=\tvs\hn'\nnh\oplus\tvs\hn''\mathrm{\ are\ called\ }\text{\it
complementary}
\end{equation}
{\it to each other}. 
As usual, we define a (full) {\it lattice\/} in $\,\tvs\hs$ to be any 
sub\-group $\,L\,$ of the additive group of $\,\tvs\hs$ generated by a basis 
of $\,\tvs\hs$ (which must consequently also be a $\,\bbZ\nh$-ba\-sis of 
$\,L$).
The quotient Lie group $\,\tvs\nnh/\nh L\,$ then is a torus, 
and we use the term {\it sub\-to\-ri\/} when referring to its compact 
connected Lie sub\-groups.
\begin{definition}\label{lsbsp}Given a lattice $\,L\,$ in $\,\tvs\nnh$, by an 
$\,L${\it-sub\-space\/} of $\,\tvs$ we will mean any vector sub\-space 
$\,\tvs\hn'$ of $\,\tvs\hs$ spanned by $\,L\cap\tvs\hn'\nnh$. One may 
equivalently require $\,\tvs\hn'$ to be the span of just a subset of $\,L$, 
rather than specifically of $\,L\cap\tvs\hn'\nnh$.
\end{definition}
\begin{lemma}\label{cplvs}
For a lattice\/ $\,L\,$ in\/ $\,\tvs\nnh$, the parallel distribution on\/
$\,\tvs\hs$ tangent to any prescribed vector sub\-space\/ $\,\tvs\hn'$
projects onto a parallel distribution\/ $\,D\hs$ on the torus group\/
$\,\tvs\nnh/\nh L$. The leaves of\/ $\,D\,$ must be 
either all compact, or all noncompact, and they are compact if and only if\/ 
$\,\tvs\hn'$ is an\/ $\,L$-sub\-space, in which case the leaf of\/ $\,D$ 
through zero is a sub\-to\-rus of\/ $\,\tvs\nnh/\nh L$.
\end{lemma}
\begin{proof}The pro\-ject\-a\-bil\-i\-ty assertion is obvious from the 
general fact, here applied to the projection
$\,\tvs\nh\to\tvs\nnh/\nh L$, that pro\-ject\-a\-bil\-i\-ty of
distributions under covering projections amounts to their 
deck-trans\-for\-ma\-tion invariance. The first 
claim about the leaves of $\,D$ follows as the leaves are one another's 
translation images. For the second, let $\,\mathcal{N}\hs$ be the leaf of 
$\,D\,$ through zero. Requiring $\,\tvs\hn'$ to be (or, not to be) an 
$\,L$-sub\-space makes $\,L\cap\tvs\hn'\nnh$, by Lemma~\ref{surjc}, a 
di\-rect-sum\-mand sub\-group of $\,L\,$ spanning $\,\tvs\hn'$ or, 
respectively, yields the existence of a nonzero linear functional $\,f\,$ on 
$\,\tvs\hn'\nnh$, the kernel of which contains $\,L\cap\tvs\hn'\nnh$. In the 
former case $\,\mathcal{N}\hs$ is a factor of a prod\-uct-of-to\-ri 
decomposition of 
$\,\tvs\nnh/\nh L$, in the latter $\,f\,$ descends to 
an unbounded function on $\,\mathcal{N}\nnh$.
\end{proof}
\begin{lemma}\label{spint}
Given a lattice\/ $\,L\,$ in\/ $\,\tvs\nnh$, the span and 
intersection of any family of\/ $\,L$-sub\-spaces are\/ $\,L$-sub\-spaces. The 
same is true if one replaces the phrase `$\nh L$-sub\-spaces' with\/ 
`$\nh H\nnh$-in\-var\-i\-ant $\,L$-sub\-spaces' for any fixed group\/ $\,H\hs$ 
of linear auto\-mor\-phisms of\/ $\,\tvs\hs$ sending\/ $\,L\,$ into itself.
\end{lemma}
\begin{proof}The assertion about spans follows from the case of two 
$\,L$-sub\-spaces, obvious in turn due to the second sentence of 
Definition~\ref{lsbsp}. Next, the intersection of the family of sub\-to\-ri in 
$\,\tvs\nnh/\nh L$, arising via Lemma~\ref{cplvs} from the given 
family of $\,L$-sub\-spaces, constitutes a compact Lie sub\-group of 
$\,\tvs\nnh/\nh L$, 
so that it is the union of 
finitely many cosets of a sub\-to\-rus $\,\mathcal{N}\nnh$. Since sub\-to\-ri 
are totally geodesic relative to the trans\-la\-tion-in\-var\-i\-ant flat
\af\ connection on 
$\,\tvs\nnh/\nh L$, while the projection 
$\,\tvs\nh\to\tvs\nnh/\nh L\,$ 
is locally dif\-feo\-mor\-phic, the tangent space of $\,\mathcal{N}\hs$ at 
zero equals the intersection of the tangent spaces of the sub\-to\-ri forming 
the family, and each tangent space corresponds to an $\,L$-sub\-space from our 
family. The conclusion is now immediate from Lemma~\ref{cplvs}.
\end{proof}
\begin{remark}\label{cptfd}For a lattice $\,L\,$ in $\,\tvs\hs$ generated by a 
basis $\,e\nnh_1\w,\dots,e\nnh_n\w$ of $\,\tvs\nnh$, the translational action
of $\,L\,$ on $\,\tvs\hs$ has an obvious compact {\it fundamental domain\/} 
(a compact subset of $\,\tvs$ intersecting all orbits of $\,L$): the 
parallelepiped formed by all the combinations
$\,t\nnh_1\w e\nnh_1\w\nh+\ldots+t_n\w e\nnh_n\w$
with $\,t\nnh_1\w,\dots,t_n\w$ ranging over $\,[\hs0,1]$.
\end{remark}
The next lemma is immediate from the first part of Lemma~\ref{surjc} and 
the well-known fact \cite[Chap.\ VII, Th\'eor\`eme 2]{bourbaki} that lattices
in $\,\tvs\hs$ are precisely the same
as discrete sub\-groups of $\,\tvs\nnh$, spanning $\,\tvs\nnh$.
\begin{lemma}\label{lttce}
For a lattice\/ $\,L\,$ in\/ $\,\tvs\nh$, a vector sub\-space\/ 
$\,\tvs\hn'\nnh\subseteq\tvs\nnh$, and\/ $\,L\nh'\nh=L\cap\tvs\hn'\nnh$,
\begin{enumerate}
  \def\theenumi{{\rm\alph{enumi}}}
\item $L\nh'$ is a lattice in the vector sub\-space spanned by 
it, and
\item $L\nh'$ constitutes a di\-rect-sum\-mand sub\-group of\/ $\,L$.
\end{enumerate}
\end{lemma}
\begin{lemma}\label{ratss}
Let\/ $\,\tws\hs$ be the rational vector sub\-space 
of a fi\-nite-di\-men\-sion\-al real vector space\/ $\,\tvs\nnh$, spanned by 
a fixed lattice\/ $\,L\,$ in\/ $\,\tvs\nnh$. Then the four sets consisting,
respectively, of all
\begin{enumerate}
  \def\theenumi{{\rm\roman{enumi}}}
\item $L$-sub\-spaces\/ $\,\tvs\hn'\nh$ of\/ $\,\tvs\nnh$,
\item di\-rect-sum\-mand sub\-groups\/ $\,L\nh'$ of\/ $\,L$,
\item rational vector sub\-spaces\/ $\,\tws\hh'\nh$ of\/ $\,\tws\nnh$,
\item sub\-to\-ri\/ $\,\mathcal{N}\hn'\nnh$ of the torus group\/ 
$\,\tvs\nnh/\nh L$, that is, its compact connected Lie sub\-groups,
\end{enumerate}
stand in mutually consistent, natural bijective correspondences with one 
another, obtained by declaring\/ $\,\tvs\hn'\nh$ to be the real span of both\/ 
$\,L\nh'$ and\/ $\,\tws\hh'\nh$ as well as the identity component of the 
pre\-im\-age of\/ $\,\mathcal{N}\hn'\nh$ under\/ the projection 
homo\-mor\-phism\/ $\,\tvs\nh\to\tvs\nnh/\nh L$. Furthermore, 
$\,\tws\hh'\nh$ equals\/ $\,\tws\cap\tvs\hn'\nnh$ and, simultaneously, is the 
rational span of\/ $\,L\nh'\nnh$, while\/ 
$\,\mathcal{N}\hn'\nh=\tvs\hn'\nh\nnh/\nh L\nh'\nh$ and\/ 
$\hs L\nh'\nh=L\cap\tvs\hn'\nh=L\cap\tws\hh'\nnh$. 
Finally, 
$\,\dimr\tvs\hn'\nh=\dimz L\nh'\nh=\dimq\tws\hh'\nh=\dim\mathcal{N}\hn'\nnh$.

`Mutual consistency' means here that the above finite set of bijections is 
closed under the operations of composition and inverse.
\end{lemma}
\begin{proof}The mappings (ii) $\to$ (i) and (iii) $\to$ (i), as well as (iv) 
$\to$ (i), defined in the three lines following (iv), are all bijections, 
with the inverses given by$\,(L\nh'\nnh,\tws\hh'\nnh,\mathcal{N}\hn')
=(L\cap\tvs\hn'\nnh,\tws\cap\tvs\hn'\nnh,\tvs\hn'\nh\nnh/\nh L\nh')$. Namely, 
each of the three mappings and their purported inverses takes values in the 
correct set, and each of the six map\-ping-in\-verse compositions is the 
respective identity. To be specific, the claim about the values follows from 
Lemma~\ref{cplvs} for (iv) $\to$ (i) and (i) $\to$ (iv), from 
Definition~\ref{lsbsp} and Lemma~\ref{surjc} for (ii) $\to$ (i) and (i) 
$\to$ (ii), while it is obvious for (i) $\to$ (iii) and, for (iii) $\to$ (i), 
immediate from Definition~\ref{lsbsp}, since we are free to 
assume that
\begin{equation}\label{ide}
(L,\,\tws\nh,\,\tvs)\,\,=\,\,(\bbZ^n\nh,\,\bbQ^n\nh,\,\rn)\hh,\hskip12pt
\mathrm{where\ \ }\,n=\dim\tvs\nh,
\end{equation}
and every rational vector sub\-space of $\,\bbQ^n$ has a basis contained in 
$\,\bbZ^n\nnh$. Next, the compositions (ii) $\to$ (i) $\to$ (ii) and (i) $\to$ 
(ii) $\to$ (i) are the identity mappings -- the former due to the fact that 
$\,L\cap\hs\spr\,L\nh'\subseteq L\nh'$ (which one sees extending a 
$\,\bbZ\nh$-ba\-sis of $\,L\nh'$ to a $\,\bbZ\nh$-ba\-sis of $\,L$) -- 
the opposite inclusion being obvious; the latter, as Definition~\ref{lsbsp} 
gives $\,\tvs\hn'=\spr\,(L\cap\tvs\hn')$. Similarly for (iii) $\to$ (i) $\to$ 
(iii) and (i) $\to$ (iii) $\to$ (i), as long as one replaces the letters 
$\,L\,$ and $\,\bbZ\,$ with $\,\tws\hs$ and $\,\bbQ$, using (\ref{ide}) and 
the line following it. Finally, (iv) $\to$ (i) $\to$ (iv) and (i) $\to$ (iv) 
$\to$ (i) are the identity mappings as a consequence of Lemma~\ref{cplvs}, 
and the dimension equalities become obvious if one, again, chooses a 
$\,\bbZ\nh$-ba\-sis of $\,L\,$ containing a $\,\bbZ\nh$-ba\-sis of 
$\,L\nh'\nnh$.
\end{proof}
In the next theorem, as $\,H\,$ is finite, all $\,A\in H\,$ must have
$\,\mathrm{det}\hskip1.3ptA=\pm\nh1$, and so the $\,L$-pre\-serv\-ing property 
of $\,H\,$ means that $\,AL=L\,$ (rather than just $\,AL\subseteq L$).
\begin{proposition}\label{invcp}
For a lattice\/ $\,L\,$ in a fi\-nite-di\-men\-sion\-al real vector space\/ 
$\,\tvs\nnh$, a finite group\/ $\,H$ of\/ $\,L$-pre\-serv\-ing linear 
auto\-mor\-phisms of\/ $\,\tvs\nnh$, and an\/ $\,H\nnh$-in\-var\-i\-ant\/ 
$\,L$-sub\-space\/ $\,\tvs\hn'$ of\/ $\,\tvs\nnh$, there exists an\/ 
$\,H\nnh$-in\-var\-i\-ant\/ $\,L$-sub\-space\/ $\,\tvs\hn''\nnh$ of\/ 
$\,\tvs\nnh$, complementary to\/ $\,\tvs\hn'$ in the sense of\/ 
{\rm(\ref{cpl})}.
\end{proposition}
\begin{proof}Let $\,\tws\hh'=\hs\tws\cap\tvs\hn'\nnh$, where $\,\tws\hs$ is
the rational span of $\,L\,$ (see Lemma~\ref{ratss}). Restricted to 
$\,\tws\nnh$, the elements of $\,H\,$ act by conjugation on the rational 
\af\ space $\,\mathcal{P}\hs$ of all $\,\bbQ$-lin\-e\-ar projections 
$\,\tws\to\tws\hh'$ (by which we mean linear operators $\,\tws\to\tws\hh'$ 
equal to the identity on $\,\tws\hh'$). The average of any orbit of the action 
of $\,H\,$ on $\,\mathcal{P}\hs$ is an $\,H\nh$-in\-var\-i\-ant 
projection $\,\tws\to\tws\hh'$ with a kernel $\,\tws\hh''$ corresponding via 
Lemma~\ref{ratss} to our required $\,\tvs\hn''\nnh$.
\end{proof}
\begin{corollary}\label{dcomp}
If\/ $\,L,\tvs\nnh,H\,$ satisfy the hypotheses of
Proposition\/~{\rm\ref{invcp}}, then 
every nonzero\/ $\,H\nnh$-in\-var\-i\-ant\/ $\,L$-sub\-space\/ 
$\,\tvs\hn'_{\hskip-2.6pt0}$ of\/ $\,\tvs$ can be decomposed into a direct sum 
of one or more nonzero\/ $\,H\nnh$-in\-var\-i\-ant\/ $\,L$-sub\-spaces, each 
of which is minimal in the sense of not containing any further nonzero 
proper\/ $\,H\nnh$-in\-var\-i\-ant\/ $\,L$-sub\-space.
\end{corollary}
\begin{proof}Induction on the possible values of 
$\,\dim\tvs\hn'_{\hskip-2.6pt0}$. The case $\,\dim\tvs\hn'_{\hskip-2.6pt0}=1\,$
is trivial. Assuming the claim true for sub\-spaces of 
dimensions less than $\,\dim\tvs\hn'_{\hskip-2.6pt0}$, along with 
non-min\-i\-mal\-ity of $\,\tvs\hn'_{\hskip-2.6pt0}$, we fix a nonzero proper 
$\,H\nh$-in\-var\-i\-ant $\,L$-sub\-space $\,\tvs\hn'$ of $\,\tvs\nnh$, 
contained in $\,\tvs\hn'_{\hskip-2.6pt0}$, and choose an
$\,H\nh$-in\-var\-i\-ant complement $\,\tvs\hn''$ 
of $\,\tvs\hn'\nnh$, guaranteed to exist by Proposition~\ref{invcp}. Since 
$\,\tvs\hn''$ intersects every coset of $\,\tvs\hn'$ in $\,\tvs\nnh$, 
including cosets within $\,\tvs\hn'_{\hskip-2.6pt0}$, the sub\-space 
$\,\tvs\hn'_{\hskip-2.6pt0}\cap\tvs\hn''$ is an $\,H\nh$-in\-var\-i\-ant 
complement of $\,\tvs\hn'$ in $\,\tvs\hn'_{\hskip-2.6pt0}$, as well as an 
$\,L$-sub\-space (due to Lemma~\ref{spint}). We may now apply the inductive 
assumption to both $\,\tvs\hn'$ and 
$\,\tvs\hn'_{\hskip-2.6pt0}\cap\tvs\hn''\nnh$.
\end{proof}
For geometric consequences of Corollary~\ref{dcomp}, see the end of 
Sect.~\ref{gc}, where we also point out that a decomposition into minimal
summands is in general nonunique.

Given a lattice $\,L\,$ in a 
fi\-nite-di\-men\-sion\-al real vector space $\,\tvs\hs$ and an 
$\,L$-sub\-space $\,\tvs\hn'$ of $\,\tvs\nnh$, the restriction to $\,L\,$ of 
the quo\-tient-space projection $\,\tvs\nnh\to\tvs\nh/\tvs\hn'$ has the 
kernel $\,L\nh'\nh=L\cap\tvs\hn'\nnh$, and so it descends to an injective 
group homo\-mor\-phism $\,L/\nh L\nh'\nnh\to\tvs\nh/\tvs\hn'\nnh$, the image 
of which is a (full) lattice in an $\,\tvs\nh/\tvs\hn'$ (which follows if 
one uses a $\,\bbZ\nh$-ba\-sis of $\,L\,$ containing a $\,\bbZ\nh$-ba\-sis of 
$\,L\nh'$). From now on we will treat $\,L/\nh L\nh'$ as a subset of 
$\,\tvs\nh/\tvs\hn'\nnh$. The discreteness of the lattice 
$\,L/\nh L\nh'\nnh\subseteq\tvs\nh/\tvs\hn'$ clearly implies the existence 
of an open subset $\,\mathcal{U}'$ of $\,\tvs\nnh$, containing 
$\,\tvs\hn'$ and forming a union of cosets of $\,\tvs\hn'\nnh$, such that 
$\,L\hn\cap\nh\mathcal{U}'\nh=\hs L\nh'\nnh$.

\section{Af\-fine spaces}\label{as}
In this section all the \af\ and vector spaces are real and
fi\-nite-di\-men\-sion\-al, 
we denote by $\,\mathrm{End}\,\tvs\hs$ the space of linear en\-do\-mor\-phisms 
of a given real vector space $\,\tvs\nnh$, and scalars stand for the
corresponding multiples of identity, so that the identity itself becomes 
$\,1\in\mathrm{End}\,\tvs\nnh$.

For an \af\ space $\,\mathcal{E}\hs$ with the translation vector space
$\,\tvs\nnh$, let $\,\mathrm{Aut}\,\mathcal{E}\hs$ be the group of all
\af\ transformations (auto\-mor\-phisms) of $\,\mathcal{E}$, and 
$\,\operatorname{Aff}\hs\mathcal{E}\hs$ the set of all (possibly
non\-bi\-ject\-ive) \af\ mappings $\,\mathcal{E}\to\mathcal{E}$. We
have the inclusions
$\,\mathrm{Aut}\,\mathcal{E}\subseteq\operatorname{Aff}\hs\mathcal{E}\hs$ and 
$\,\mathcal{V}\hh\subseteq\mathrm{Aut}\,\mathcal{E}$, the latter
expressing the fact that $\,\mathrm{Aut}\,\mathcal{E}\hs$ contains the
normal sub\-group consisting of all translations. Any vector sub\-space
$\,\tvs\hn'$ of $\,\tvs\hs$ gives rise to a foliation of $\,\mathcal{E}$,
with the leaves formed by \af\ sub\-spaces $\,\mathcal{E}\nh'$ parallel to
$\,\tvs\hn'\nnh$, meaning
\begin{equation}\label{orb}
\mathrm{the\ orbits\ of\ the\ translational\ action\ of\ 
}\,\tvs\hn'\mathrm{\ on\ }\,\mathcal{E}\hh,
\end{equation}
which will also be referred to as the {\it cosets\/} of 
$\,\tvs\hn'$ in $\,\mathcal{E}$. The resulting leaf (quotient) space 
$\,\mathcal{E}\hn/\hn\tvs\hn'$ constitutes an \af\ space having the 
translation vector space $\,\tvs\nnh/\hn\tvs\hn'\nnh$. Clearly, for cosets
$\,\mathcal{E}\nh'$ and $\,\mathcal{E}\nh''$ of vector sub\-spaces 
$\,\tvs\hn'\nnh,\tvs\hn''$ in a vector space $\,\tvs\nh$,
\begin{equation}\label{opt}
\mathrm{the\ complementarity\ condition\ (\ref{cpl})\ implies\ that\
}\,\mathcal{E}\nh'\nnh\cap\hs\mathcal{E}\nh''\nh\mathrm{\ is\ a\ 
one}\hyp\mathrm{point\ set.}
\end{equation}
A fixed inner product in $\,\hs\tvs\hs$ turns $\,\mathcal{E}\hs$ into a 
{\it Euclidean \af\ space}, with the isom\-e\-try group 
$\,\mathrm{Iso}\,\mathcal{E}\subseteq\mathrm{Aut}\,\mathcal{E}$. If 
$\,\rd\in(0,\infty)$, we define the $\,\rd${\it-neigh\-bor\-hood\/} of 
an \af\ sub\-space $\,\mathcal{E}\nh'$ of $\,\mathcal{E}\hs$ to be the set of 
points in $\,\mathcal{E}\hs$ lying at distances less that $\,\rd\,$ from 
$\,\mathcal{E}\nh'\nnh$. Clearly, the $\,\rd$-neigh\-bor\-hood of 
$\,\mathcal{E}\nh'$ is a union of cosets of a vector sub\-space $\,\tvs\hn'$ of 
$\,\tvs\hs$ (one of them being $\,\mathcal{E}\nh'$ itself), as well as the 
preimage, under the projection 
$\,\mathcal{E}\to\mathcal{E}\hn/\hn\tvs\hn'\nnh$, of the radius $\,\rd\,$ 
open ball centered at the point $\,\mathcal{E}\nh'$ in the quotient Euclidean 
\af\ space $\,\mathcal{E}\hn/\hn\tvs\hn'$ (for the obvious inner product 
on $\,\tvs\nnh/\hn\tvs\hn'$).

Given a Euclidean \af\ space 
$\,\mathcal{E}\hs$ with the translation vector space $\,\tvs\hs$ and an
af\-fine sub\-space $\,\mathcal{E}\nh'\subseteq\mathcal{E}\hs$ parallel, as in
(\ref{orb}), to a 
vector sub\-space $\,\tvs\hn'\subseteq\tvs\nh$, (\af)
self-isom\-e\-tries $\,\zeta\,$ of $\,\mathcal{E}$ such that
$\,\zeta(x)=x\,$ for all $\,x\in\mathcal{E}\nh'$ are in an obvious one-to-one
correspondence with linear self-isom\-e\-tries $\,A\,$ of the orthogonal
complement of $\,\tvs\hn'\nnh$. In this case, for easy later reference (in the
proof of Lemma~\ref{cmmut}),
\begin{equation}\label{afx}
\mathrm{we\ will\ call\ }\,\zeta\,\mathrm{\ the\ \af\ extension\ of\
}\,A\,\mathrm{\ centered\ on\ }\,\mathcal{E}\nh'\nh.
\end{equation}
\begin{definition}\label{trprt}In an \af\ space $\,\mathcal{E}$ having 
the translation vector space $\,\tvs\nnh$, given an \af\ mapping
$\,\gamma\in\operatorname{Aff}\hs\mathcal{E}\hs$ with the linear part 
$\,A\in\mathrm{End}\,\tvs\nh$, we define the {\it
trans\-la\-tion\-al-part coset\/} of $\,\gamma\,$ to be the sub\-set
$\,b\hs+\nh\hat{\tvs}\hs$ of $\,\tvs\nnh$, where $\,\hat{\tvs}$ denotes the
image of $\,A-1$, and $\,b\in\mathcal{V}\hs$ is the translational part of
$\,\gamma\,$ relative to a fixed origin $\,o\in\mathcal{E}$, in the sense that 
$\,\gamma(o+v)=o+Av\nh+\hn b\,$ for all $\,v\in\tvs\nnh$. The coset 
$\,b\hs+\nh\hat{\tvs}\hs$ is clearly independent of the choice of an origin
$\,o$, as a new origin $\,o+w\,$ results in the replacement of $\,b$ with 
$\,b+(A-1)w$.
\end{definition}
For an \af\ transformation 
$\,\gamma\in\mathrm{Aut}\,\mathcal{E}$ of an \af\ space 
$\hs\mathcal{E}$ with the translation vector space $\tvs\nnh$, and a
vector sub\-space $\,\tvs\hn'$ of $\,\tvs\nh$, consider this condition:
\begin{equation}\label{cnd}
\begin{array}{l}
\mathrm{the\ linear\ part\ }\,\,A\,\,\mathrm{\ of\ 
}\,\hs\gamma\hs\,\mathrm{\ leaves\ }\,\,\tvs\hn'\hs\mathrm{\ invariant\ and\
descends\ to}\\
\mathrm{the\ identity\ transformation\ of\
}\hs\tvs\nnh/\hn\tvs\hn'\nnh\mathrm{,\nnh\ that\ is,\
}(\hn A\nh-\nnh1\hn)(\hn\tvs)\nnh\subseteq\nnh\tvs\hn'\nnh.
\end{array}
\end{equation}
\begin{lemma}\label{fltmf}
If, in Lemma\/~{\rm\ref{nrexp}}, $\,\mathcal{M}\hn'$ is a compact 
leaf of a parallel distribution\/ $\,D\,$ on a complete flat Riemannian
manifold\/ $\,\mathcal{M}$, we get the following additional
conclusions.
\begin{enumerate}
  \def\theenumi{{\rm\alph{enumi}}}
\item Every level of\/ $\,\mathrm{dist}\hh(\mathcal{M}\hn'\nh,\,\cdot\,)\,$ 
in\/ $\,\mathcal{M}'\hskip-4.1pt_\rd\w$, and\/ $\,\mathcal{M}'\hskip-4.1pt_\rd\w$ itself, is a 
union of leaves of\/ $\,D$.
\item Restrictions of\/ 
$\,\mathcal{M}'\hskip-4.1pt_\rd\w\ni x\mapsto y\in\mathcal{M}\hn'$ 
to leaves of\/ $\,D\,$ in\/ $\,\mathcal{M}'\hskip-4.1pt_\rd\w$ are locally isometric.
\item The local inverses of all the above lo\-cal\-ly-i\-so\-met\-ric 
restrictions correspond via the dif\-feo\-mor\-phism\/ $\,\mathrm{Exp}^\perp$ 
to all local sections of the normal bundle of\/ $\,\mathcal{M}\hn'$ obtained 
by restricting to\/ $\,\mathcal{M}\hn'$ local parallel vector fields of 
lengths\/ $\,r\in[\hs0,\rd)\,$ that are tangent to\/ $\,\mathcal{M}\,$ and 
normal to\/ $\,\mathcal{M}\hn'\nnh$, with\/ $\,\hs r\hskip.55pt$ equal to the 
value of\/ $\,\mathrm{dist}\hh(\mathcal{M}\hn'\nh,\,\cdot\,)\,$ on the leaf in
question.
\end{enumerate}
\end{lemma}
This trivially follows from the fact the  pull\-back of $\,D\,$ to the 
Euclidean \af\ space $\,\mathcal{E}$ constituting the Riemannian universal 
covering space of $\,\mathcal{M}\,$ is a distribution
whose integral manifolds are the \af\ sub\-spaces parallel to 
$\,\tvs\hn'\nnh$, in the sense of (\ref{orb}), for some vector 
sub\-space $\,\tvs\hn'$ of the translation vector space $\,\tvs\hs$ of 
$\,\mathcal{E}$.

\section{Bieberbach groups and flat manifolds}\label{bg}
Let $\,\mathcal{E}\hs$ be a Euclidean \af\ $\,n$-space 
with the translation vector space $\,\tvs\nnh$. By a {\it Bie\-ber\-bach 
group\/} \cite[p.\ 4, Definition 1.7]{charlap} in $\,\mathcal{E}\hs$ one means 
any tor\-sion-free discrete 
sub\-group $\,\bg\,$ of $\,\mathrm{Iso}\,\mathcal{E}\hs$ for which there 
exists a compact fundamental domain (Remark~\ref{cptfd}).
Using the lin\-e\-ar-part homo\-mor\-phism
$\,\lp:\mathrm{Aut}\,\mathcal{E}\to\mathrm{Aut}\,\hn\tvs 
\cong\mathrm{GL}\hs(n,\bbR)$, one
defines the {\it 
lattice sub\-group\/} $\,L$ of $\,\bg\,$ and its {\it holonomy group\/} 
$\,H\subseteq\mathrm{Iso}\hs\tvs\cong\mathrm{O}\hs(n)\,$ by
\begin{equation}\label{lah}
L\,=\,\bg\hn\cap\tvs,\hskip22ptH\,=\,\lp(\bg)\hh.
\end{equation}
Thus, $\,L\,$ is the set of all translations lying in $\,\bg\,$ (which also 
makes it the kernel of the restriction $\,\lp:\bg\to H$), and $\,H\,$ 
consists of the linear parts of elements of $\,\bg$. Note that 
$\,L\subseteq\tvs\hs$ is a (full) lattice in the usual sense
\cite[p.\ 17, Theorem 3.1(ii)]{charlap},
as defined in Sect.~\ref{lv}. The relations involving 
$\,\bg,L$ and $\,H$ are conveniently summarized by the short exact sequence
\begin{equation}\label{exa}
L\,\to\,\hs\bg\,\to\,H,\hskip8pt\mathrm{where\ the\ arrows\ are\ the\ 
inclusion\ homo\-mor\-phism\ and\ }\,\lp\hh.
\end{equation}
As the normal sub\-group $\,L\,$ of $\,\bg\,$ is 
A\-bel\-i\-an, the action of $\,\bg\,$ on $\,L$ by conjugation descends to 
an action on $\,L\,$ of the quotient group $\,\bg/L$, identified via 
(\ref{exa}) with $\,H$.
\begin{equation}\label{cnj}
\begin{array}{l}
\mathrm{This\ last\ action\ clearly\ coincides\ with\ the\ ordinary\ linear\
action\ of}\\
H\,\mathrm{\ on\ }\hs\tvs\nnh\mathrm{,\ restricted\ to\ the\ lattice\
}\hs\,L\hs\subseteq\tvs\nnh\mathrm{,\ and\ so\ }\,L\hs\mathrm{\ is\
}\,H\nh\hyp\mathrm{in\-var\-i\-ant.}
\end{array}
\end{equation}
\begin{remark}\label{frpdc}The action of a Bie\-ber\-bach group $\,\bg\,$ on 
the Euclidean \af\ space $\,\mathcal{E}$ being always free and properly 
discontinuous \cite[p.\ 3, Proposition 1.1]{charlap}, the quotient
$\,\mathcal{M}=\mathcal{E}/\hh\bg\hh$, with the 
projected metric, forms a compact flat Riemannian manifold, while $\,H\,$ must 
be finite \cite[p.\ 17, Theorem 3.1(i)]{charlap}.
\end{remark}
\begin{remark}\label{bijct}The assignment of 
$\,\mathcal{M}=\mathcal{E}/\hh\bg\,$ to $\,\bg\,$ establishes a well-known 
bijective correspondence \cite[p.\ 65, Theorem 5.4]{charlap} between
equivalence classes of 
Bie\-ber\-bach groups and isometry types of compact flat Riemannian manifolds. 
Bie\-ber\-bach groups $\,\bg\,$ and $\,\hat\bg\,$ in Euclidean \af\ 
spaces $\,\mathcal{E}\hs$ and $\,\hat{\mathcal{E}}\hs$ are called {\it 
equivalent\/} here if some \af\ isometry 
$\,\mathcal{E}\to\hat{\mathcal{E}}\hs$ conjugates $\,\bg$ onto $\,\hat\bg$. 
Furthermore, $\,\bg\,$ and $\,H\,$ in (\ref{exa}) serve as the fundamental 
and holonomy groups of $\,\mathcal{M}$, while $\,\bg\,$ also acts via deck 
tranformations on the Riemannian universal covering space of $\,\mathcal{M}$, 
isometrically identified with $\,\mathcal{E}$.
\end{remark}
Since the lattice sub\-group $\,L\,$ of the fundamental group $\,\bg\,$ of
$\,\mathcal{M}\,$ gives rise to a covering projection 
$\,\mathcal{E}/\nh L\to\mathcal{M}$, Lemmas~\ref{cplvs} and~\ref{spint}
combined with Remark~\ref{bijct} have the following obvious consequence.
\begin{corollary}\label{clsed}
In any compact flat Riemannian manifold, the class of parallel distributions
with compact leaves is closed under spans and intersections.
\end{corollary}

\section{Lat\-tice-re\-duc\-i\-bil\-i\-ty}\label{lr}
A Bie\-ber\-bach group $\,\bg\,$ in a Euclidean \af\ space 
$\,\mathcal{E}\hs$ (or, the compact flat Riemannian manifold 
$\,\mathcal{M}=\mathcal{E}/\hh\bg\,$ corresponding to $\,\bg$, in the sense of 
Remark~\ref{bijct}, will be called {\it lat\-tice-re\-duc\-i\-ble\/} if, 
for $\,\tvs\nnh,H\,$ and $\,L\,$ associated with $\,\mathcal{E}\hs$ and 
$\,\bg\,$ as in Sect.~\ref{bg}, there exists $\,\tvs\hn'$ such that
\begin{equation}\label{lrd}
\,\tvs\hn'\mathrm{\ is\ a\ nonzero\ proper\ 
}\,H\nh\hyp\mathrm{in\-var\-i\-ant\ }\,L\hyp\mathrm{subspace\ of\ 
}\,\tvs\nnh.
\end{equation}
(See Definition~\ref{lsbsp}.) To emphasize the role of $\,\tvs\hn'$ in 
(\ref{lrd}), we also say that
\begin{equation}\label{qdr}
\mathrm{the\ lat\-tice}\hyp\mathrm{re\-duc\-i\-bil\-i\-ty\ condition\ 
(\ref{lrd})\ holds\ for\ 
}\hs(\tvs\nnh,H,L,\nnh\tvs\hn')\hh.
\end{equation}
As shown by Hiss and Szcze\-pa\'n\-ski 
\cite{hiss-szczepanski}, {\it every compact flat Riemannian manifold 
of dimension greater than one is lat\-tice-re\-duc\-i\-ble}. For details,
see the Appendix.

Given a Bie\-ber\-bach group $\,\bg\,$ in a Euclidean \af\ space 
$\,\mathcal{E}\hs$ and an \af\ sub\-space $\,\mathcal{E}\nh'$ of 
$\,\mathcal{E}\hs$ parallel, as in (\ref{orb}), to a vector sub\-space
$\,\tvs\hn'$ of its 
translation space $\,\tvs\nh$, satisfying (\ref{lrd}) -- (\ref{qdr}),
we denote by $\,\stb\hn'$ the {\it stabilizer group of\/}
$\,\mathcal{E}\nh'$ {\it 
in\/} $\,\bg\hh$, so that
\begin{equation}\label{stb}
\stb\hn'\mathrm{\ consists\ of\ all\ the\ elements\ of\ }\,\bg\hs\mathrm{\ 
mapping\ }\,\mathcal{E}\nh'\mathrm{\ onto\ itself.} 
\end{equation}
Let $\,\gamma\in\bg$. As the foliation of $\,\mathcal{E}\hs$ formed by the
cosets of $\,\tvs\hn'$ is $\,\bg$-in\-var\-i\-ant, cf.\ (\ref{lrd}),
\begin{equation}\label{iff}
\gamma\in\stb\hn'\mathrm{\ if\ and\ only\ if\ 
}\,\gamma(\mathcal{E}\nh')\,\mathrm{\ intersects\ }\,\mathcal{E}\nh'\nh.
\end{equation}
\begin{theorem}\label{restr}
For a lat\-tice-re\-duc\-i\-ble Bie\-ber\-bach 
group\/ $\,\bg\,$ in a Euclidean \af\ space\/ $\,\mathcal{E}\hs$ and a 
vector subspace\/ $\,\tvs\hn'$ of\/ $\,\tvs\,$ with\/ {\rm(\ref{lrd})}, the 
following conclusions hold.
\begin{enumerate}
  \def\theenumi{{\rm\roman{enumi}}}
\item The \af\ sub\-spaces of dimension\/ $\,\dim\tvs\hn'$ in\/ 
$\,\mathcal{E}$, parallel to\/ $\,\tvs\hn'$ in the sense of\/
{\rm(\ref{orb})}, are the leaves of a foliation\/
$\,F\hskip-3pt_{\mathcal{E}}\w$ on\/ $\,\mathcal{E}$,   
pro\-ject\-a\-ble under the covering projections\/ 
$\,\mathrm{pr}:\mathcal{E}\to\mathcal{M}=\mathcal{E}/\hh\bg\,$ and\/ 
$\,\mathcal{E}\to\mathcal{T}=\,\mathcal{E}/\nh L\,$ onto foliations\/ 
$\,F\hskip-3.8pt_{\mathcal{M}}\w\,$ of\/ $\,\mathcal{M}\,$ and\/ 
$\,F\hskip-3pt_{\mathcal{T}}\w$ of the torus\/ 
$\,\mathcal{T}=\,\mathcal{E}/\nh L$, both of which have compact totally 
geodesic leaves, tangent to a parallel distribution.
\item The leaves\/ $\,\mathcal{M}\hn'$ of\/ 
$\,F\hskip-3.8pt_{\mathcal{M}}\w\hs$ 
coincide with the\/ $\,\mathrm{pr}$-im\-ages of the
leaves\/ $\,\mathcal{E}\nh'$ of\/ 
$\,F\hskip-3pt_{\mathcal{E}}\w$, and the restrictions\/ 
$\,\mathrm{pr}:\mathcal{E}\nh'\nh\to\mathcal{M}\hn'$ are covering projections. 
The same remains true if one replaces $\,\mathcal{M}\,$ and\/ 
$\,\mathrm{pr}\,$ with\/ $\,\mathcal{T}\,$ and the projection\/ 
$\,\mathcal{E}\to\mathcal{T}\nnh$. Any such\/ $\,\mathcal{M}\hn'\nnh$, being a 
compact flat Riemannian manifold, corresponds via Remark\/~{\rm\ref{bijct}} to 
a Bie\-ber\-bach group\/ $\,\bg\hn'$ in the Euclidean \af\ space\/ 
$\,\mathcal{E}\nh'\nnh$. For\/ $\,L\nh'\nnh,H\nh'$ appearing in the\/
$\,\mathcal{M}\hn'\nnh$-an\-a\-log\/ 
$\,L\nh'\nh\to\hs\bg\hn'\nh\to H\nh'$ of\/ {\rm(\ref{exa})},
with\/ $\,\stb\hn'$ defined by\/ {\rm(\ref{stb})},
\begin{enumerate}
  \def\theenumi{{\rm\alph{enumi}}}
\item $\bg\hn'$ consists of the restrictions to\/ $\,\mathcal{E}\nh'$ of all the 
elements of\/ $\,\stb\hn'\nnh$,
\item $H\nh'$ is formed by the restrictions to\/ $\,\tvs\hn'$ of the linear 
parts of elements of \/ $\,\stb\hn'\nnh$,
\item $L\nh'\nh=\bg\hn'\nh\cap\tvs\hn'\nnh$, as in\/ {\rm(\ref{lah})}, and\/ 
$\,L\cap\tvs\hn'\subseteq L\nh'\nnh$.
\end{enumerate}
\item The restriction homo\-mor\-phism\/ $\,\stb\hn'\nnh\to\bg\hn'$ of\/
{\rm(ii-a)} is an iso\-mor\-phism.
\end{enumerate}
\end{theorem}
The last inclusion of (ii-c) may be proper; see the end of Sect.~\ref{gk}.

We chose to format the proof Theorem~\ref{restr} as the whole next section,
since some parts of it are of independent interest, and can in this way be
more comfortably cited later in the text (Sections~\ref{gi} -- \ref{gg},
\ref{ig} -- \ref{ih}, \ref{rh}).

\section{Proof of Theorem~\ref{restr}}\label{pt}
The projectability of the foliation $\,F\hskip-3pt_{\mathcal{E}}\w$ under both 
covering projections 
$\,\mathrm{pr}:\mathcal{E}\to\mathcal{M}\,$ and 
$\,\mathcal{E}\to\mathcal{T}\,$ follows as a trivial consequence from the fact 
that, due to the $\,H\nh$-in\-var\-i\-ance of $\,\tvs\hn'\nnh$,
\begin{equation}\label{inv}
F\hskip-3pt_{\mathcal{E}}\w\,\mathrm{\ is\ 
}\,\bg\nh\hyp\mathrm{in\-var\-i\-ant\ and, obviously,\ 
}\,L\hyp\mathrm{in\-var\-i\-ant,}
\end{equation}
while Lemma~\ref{cptlv}(ii) implies the integrability of the image
distribution. 
Next,
\begin{equation}\label{cmp}
\begin{array}{l}
\mathrm{pr}\nh\,\mathrm{\ is\ the\ composition\ 
}\,\mathcal{E}\hn\to\nh\mathcal{T}\hskip-3pt
\to\nh\mathcal{M}\nh\,\mathrm{\ 
of\ two\ mappings\hskip-3pt:\hskip.8ptthe}\\
\mathrm{u\-ni\-ver\-sal}\hyp\mathrm{cov\-er\-ing\hskip1pt\
projection\hskip1pt\ of\hskip1pt\ the\hskip1pt\ 
flat\hskip1pt\ torus\hskip.3pt\ }\,\hs\mathcal{T}\hn=\,\mathcal{E}/\nh L,\\
\mathrm{and\hs\hs\ the\hs\hs\ quotient\hs\hs\ projection\hs\hs\ for\hs\hs\
the\hs\hs\ action\hs\hs\ of\hs\ 
}\,\,\hs\bg\hs\,\,\mathrm{\hs\hs\ on\hs\hs\ }\,\,\mathcal{T}\nnh,
\end{array}
\end{equation}
the latter action clearly becoming free if one replaces $\,\bg\,$ with 
$\,\bg/\nh L\cong\nh H$. Both factor mappings, 
$\,\mathcal{E}\to\mathcal{T}\,$ and 
$\,\mathcal{T}\nh\to\mathcal{M}$, are covering projections -- the first 
since $\,L\,$ is a lattice in $\,\tvs\nh$, the second due to 
Remark~\ref{covpr}(a). Parts (iii)\hs--\hs(iv) of Lemma~\ref{cptlv}, along 
with Lemma~\ref{cplvs}, may now be applied to the foliations 
$\,F\hskip-3pt_{\mathcal{T}}\w$ and $\,F\hskip-3.8pt_{\mathcal{M}}\w$ of the 
torus $\,\mathcal{T}\,$ and of $\,\mathcal{M}\,$ obtained as projections of 
$\,F\hskip-3pt_{\mathcal{E}}\w$, proving the last (com\-pact-leaves) claim of 
(i), as well as the first two sentences of (ii).

We now fix a leaf $\,\mathcal{E}\nh'$ of $\,F\hskip-3pt_{\mathcal{E}}\w$, 
and choose a leaf $\,\mathcal{M}\hn'$ of $\,F\hskip-3.8pt_{\mathcal{M}}\w$ 
containing $\,\mathrm{pr}(\mathcal{E}\nh')$, cf.\ Lemma~\ref{cptlv}(i). It 
follows that
\begin{equation}\label{cpr}
\mathrm{pr}:\mathcal{E}\nh'\nh\to\mathcal{M}\hn'\,\mathrm{\ is\ a\
(surjective)\ covering\ projection,}
\end{equation}
since (\ref{cmp}) decomposes
$\,\mathrm{pr}:\mathcal{E}\nh'\nh\to\mathcal{M}\hn'$ 
into the composition 
$\,\mathcal{E}\nh'\nh\to\mathcal{T}\hh'\nnh\to\nh\mathcal{M}\hn'\nnh$, 
in which the first mapping is the u\-ni\-ver\-sal-cov\-er\-ing projection of 
the torus $\,\mathcal{T}\hh'\nh=\hs\mathcal{E}\nh'\nnh/\nh L\nh'\nnh$, and the 
second one must be a covering due to Remark~\ref{covpr}(b).

Two points of $\,\mathcal{E}\nh'$ have the same $\,\mathrm{pr}$-im\-age if and 
only if one is transformed into the other by an element of the group 
$\,\bg\hn'$ described in assertion (ii-a); namely, the `only if' part follows 
since, given $\,x,y\in\mathcal{E}\nh'$ with
$\,\mathrm{pr}(x)=\mathrm{pr}(y)\,$ 
in $\,\mathcal{M}=\mathcal{E}/\hh\bg\hh$, the element of $\,\bg$ sending 
$\,x\,$ to $\,y\,$ must lie in $\,\stb\hn'$ by (\ref{iff}). Furthermore, 
$\,\bg\hn'$ acts on $\,\mathcal{E}\nh'$ freely since $\,\bg\,$ does so on
$\,\mathcal{E}\hs$ (Remark~\ref{frpdc}). Thus, $\,\bg\hn'$ 
coincides with the deck transformation group for the u\-ni\-ver\-sal 
cov\-er\-ing projection (\ref{cpr}), and satisfies (ii-a). Next, (ii-b) and
(ii-c) are consequences of the definitions of $\,H\nh'$ and $\,L\nh'\nh$.

Finally, (iii) follows since nontrivial elements of $\,\stb\hn'\nnh$, being 
fix\-ed-point free (Remark~\ref{frpdc}), have nontrivial restrictions to 
$\,\mathcal{E}\nh'\nnh$. 

\section{Geometric consequences of Lemma~\ref{cplvs} and
Proposition~\ref{invcp}}\label{gc}
Hiss and Szcze\-pa\'n\-ski's result mentioned in the Introduction, combined
with Remark~\ref{bijct}, Proposition~\ref{invcp} 
and Theorem~\ref{restr}, has the following immediate consequence.
\begin{theorem}\label{twopr}
Every compact flat Riemannian manifold\/ $\,\mathcal{M}\hs$ of dimension\/ 
$\,n\ge2$ admits two proper parallel distributions\/ $\,D\,$ and\/  
$\,\hat D\hs$ with compact leaves, which are complementary to each other 
in the sense that\/ $\,T\nh\mathcal{M}\hh=\hs D\oplus\hat D\nh$.
\end{theorem}
Theorem~\ref{restr}(i), Lemma~\ref{cplvs} and Corollary~\ref{dcomp} also
easily imply that the tangent bundle $\,T\nh\mathcal{M}\,$ of any compact
flat Riemannian manifold $\,\mathcal{M}\,$ admits a maximal di\-rect-sum
decomposition into parallel sub\-bun\-dles (distributions) with compact
leaves, maximality meaning that none of the summand sub\-bun\-dles can be
further decomposed in the same manner. Hiss and Szcze\-pa\'n\-ski's result 
\cite{hiss-szczepanski} guarantees that, unless
$\,\dim\mathcal{M}<2$, at least two such summand distributions are present.

Decompositions just mentioned may be quite far from unique: when
$\,\mathcal{M}\,$ is an $\,n$-torus, they stand in a bijective correspondence
with decompositions of a $\,n$-di\-men\-sion\-al  rational vector space into a
direct sum of lines. This
is why one probably should not expect them to have interesting general
properties.

\section{Geometries of individual leaves}\label{gi}
Throughout this section we adopt the assumptions and notation of 
Theorem~\ref{restr}. The $\,\bg$-in\-var\-i\-ance of the foliation 
$\,F\hskip-3pt_{\mathcal{E}}\w$, cf.\ (\ref{inv}), trivially gives rise to
the obvious
\begin{equation}\label{act}
\mathrm{isometric\ action\ of\ }\,\bg\hs\mathrm{\ on\ the\ 
quotient\ Euclidean\ \af\ space\ }\,\mathcal{E}\hn/\hn\tvs\hn'
\end{equation}
(that is, on the leaf space of $\,F\hskip-3pt_{\mathcal{E}}\w$, the points of 
which coincide with the \af\ sub\-spaces $\,\mathcal{E}\nh'$ of 
$\,\mathcal{E}\hs$ parallel to $\,\tvs\hn'$). Whenever 
$\,\mathcal{E}\nh'\nnh\in\mathcal{E}\hn/\hn\tvs\hn'$ is fixed, its stabilizer
group $\,\stb\hn'$ in 
(\ref{stb}) obviously coincides with the iso\-tropy group of $\,\mathcal{E}\nh'$  
for (\ref{act}). The action (\ref{act}) is not effective, as the kernel of the 
corresponding homo\-mor\-phism 
$\,\bg\to\mathrm{Iso}\,[\mathcal{E}\hn/\hn\tvs\hn']$ clearly contains the 
group $\,L\nh'\nh=L\cap\tvs\hn'$ forming a lattice in $\,\tvs\hn'\nnh$, cf.\ 
Definition~\ref{lsbsp} and Lemma~\ref{lttce}(a). Now the
$\,H\nh$-in\-var\-i\-ance
of $\,L\,$ -- see (\ref{cnj}) -- combined with the
$\,H\nh$-in\-var\-i\-ance of 
$\,\tvs\hn'$ shows that $\,L\nh'\nh=L\cap\tvs\hn'$ is a normal
sub\-group of $\,\bg\hh$, which leads to a further homo\-mor\-phism 
$\,\bg/\nh L\nh'\nh\to\mathrm{Iso}\,[\mathcal{E}\hn/\hn\tvs\hn']\,$ (still in 
general noninjective, cf.\ (\ref{nni}) below). Let 
$\,\mathrm{pr}\,$ again stand for the covering projection 
$\,\mathcal{E}\to\mathcal{M}=\mathcal{E}/\hh\bg$.

Given $\,\mathcal{E}\nh'\nnh\in\mathcal{E}\hn/\hn\tvs\hn'$ and a vector 
$\,v\in\tvs\hs$ orthogonal to $\,\tvs\hn'\nnh$, we set 
$\,\mathcal{M}_v'\nh=\mathrm{pr}(\mathcal{E}\nh'\nnh+v)$, so that,
according to (\ref{cpr}), 
$\,\mathcal{M}_0'=\mathcal{M}\hn'\nnh$. Again by (\ref{cpr}), 
\begin{equation}\label{unc}
\mathrm{pr}:\mathcal{E}\nh'\nnh+v\to\mathcal{M}_v'\mathrm{\ is\ a\ 
lo\-cal\-ly}\hyp\mathrm{i\-so\-met\-ric\ 
u\-ni\-ver\-sal}\hyp\mathrm{cov\-er\-ing\ projection.}
\end{equation}
and $\,\mathcal{M}_v'$ must be a (compact) leaf of
$\,F\hskip-3.8pt_{\mathcal{M}}\w$. We also
\begin{equation}\label{svp}
\begin{array}{l}
\mathrm{choose\hs\hs\ }\,\,\rd\hs\,\mathrm{\hs\hs\ as\hs\ in\hs\
Lemma~\ref{nrexp}\hs\ and\hs\ 
Lemma~\ref{fltmf}\hs\ for\hs\ the\hs\ sub\-man\-i\-fold}\\
\mathcal{M}\hn'\nnh=\hn\mathrm{pr}(\mathcal{E}\nh')\mathrm{\ with\
(\ref{cpr})\ of\ the\ compact\ flat\ manifold\ 
}\,\mathcal{M}\hn=\mathcal{E}/\hh\bg,\\
\mathrm{and\ denote\ by\
}\,\stb_{\nh v}'\hh\subseteq\hs\bg\,\mathrm{\
the\ stabilizer\ group\ of\ }\,\,\mathcal{E}\nh'+\,v\mathrm{,\ cf.\
(\ref{stb}).}
\end{array} 
\end{equation}
\begin{lemma}\label{isolv}
Under the above hypotheses, for any\/ 
$\,\mathcal{E}\nh'\nnh\in\mathcal{E}\hn/\hn\tvs\hn'$ there exists\/ 
$\,\rd\in(0,\infty)$ such that, whenever\/ $\,u\in\tvs\hs$ is a unit vector 
orthogonal to\/ $\,\tvs\hn'$ and\/ $\,r,s\in(0,\rd)$, the iso\-me\-tries\/ 
$\,\mathcal{E}\nh'\nnh+ru\to\mathcal{E}\nh'\nnh+su\,$ and 
$\,\mathcal{E}\nh'\nnh+ru\to\mathcal{E}\nh'$ acting via translations by the vectors 
$\,(s-r)u$ and, respectively, $\,-ru\hh$, descend under the 
u\-ni\-ver\-sal-cov\-er\-ing projections\/ {\rm(\ref{unc})}, with\/ 
$\,v\,$ equal to\/ $\,ru,su\,$ or\/ $\,0$, to an iso\-me\-try\/ 
$\,\mathcal{M}_{\hh r\hn u}'\nh\to\mathcal{M}_{\hh su}'$ or, respectively, a\/ 
$\,k$-fold covering projection\/ 
$\,\mathcal{M}_{\hh r\hn u}'\nh\to\mathcal{M}\hn'\nnh$, where the integer\/ 
$\,k=k(u)\ge1\,$ may depend on $\,u\,$ -- see the end of Sect.~{\rm\ref{gk}}
-- but not on\/ $\,r$.
\end{lemma}
\begin{proof}For $\,\rd\,$ selected in (\ref{svp}) and any 
$\,c\in[\hs0,1]$, let 
$\,\psi\nnh_c\w:\mathcal{M}\nh_\rd\w\to\nh\mathcal{M}\nh_\rd\w$ correspond, 
via the $\,\mathrm{Exp}^\perp\nnh$-dif\-feo\-mor\-phic identification of 
Lemma~\ref{nrexp}(a), to the mapping 
$\,\mathcal{N}\hskip-3pt_\rd\w\to\nh\mathcal{N}\hskip-3pt_\rd\w$ which 
multiplies vectors normal to $\,\mathcal{M}\hn'$ by the scalar $\,c$. With 
$\,\phi\,$ denoting our iso\-me\-try 
$\,\mathcal{E}\nh'\nnh+ru\to\mathcal{E}\nh'\nnh+su\,$ (or, 
$\,\mathcal{E}\nh'\nnh+ru\to\mathcal{E}\nh'$) we now have 
$\,\mathrm{pr}\circ\phi=\psi\nnh_c\w\nh\circ\hh\mathrm{pr}\,$ on 
$\,\mathcal{E}\nh'\nnh+ru$, where $\,c=s/r\,$ 
(or, respectively, $\,c=0$) \,since, given $\,x\in\mathcal{E}\nh'\nnh$, the 
$\,\mathrm{pr}$-im\-age of the line segment $\,\{x+tu:t\in[\hs0,\rd)\}$ in 
$\,\mathcal{E}\hs$ is the length $\,\rd\,$ minimizing geodesic segment in 
$\,\mathcal{M}'\hskip-4.1pt_\rd\w$ emanating from the point 
$\,y=\mathrm{pr}(x)\in\mathcal{M}\hn'$ in a direction normal to 
$\,\mathcal{M}\hn'\nnh$, and $\,\mathrm{pr}\circ\phi\,$ sends $\,x+tu$, in 
both cases, to $\,\mathrm{pr}(x+ctu)=\psi\nnh_c\w(\mathrm{pr}(x+tu))$. The 
$\,\mathrm{pr}$-im\-age of $\,\phi(z)$, for any $\,z\in\mathcal{E}\nh'\nnh+ru$, 
thus depends only on $\,\mathrm{pr}(z)\,$ (by being its 
$\,\psi\nnh_c\w\nh$-im\-age), and so both original isometries $\,\phi\,$ 
descend to (necessarily lo\-cal\-ly-i\-so\-met\-ric) mappings 
$\,\mathcal{M}_{\hh r\hn u}'\nh\to\mathcal{M}_{\hh su}'$ and 
$\,\mathcal{M}_{\hh r\hn u}'\nh\to\mathcal{M}\hn'\nnh$, which constitute finite 
coverings (Remark~\ref{covpr}(b)). The former is also bi\-jec\-tive, 
its inverse arising when one switches $\,r\hs$ and $\,s$. As the composition  
$\,\mathcal{M}_{\hh su}'\nh\to\mathcal{M}_{\hh r\hn u}'\nh\to\mathcal{M}\hn'$ 
clearly equals the analogous covering projection 
$\,\mathcal{M}_{\hh su}'\nh\to\mathcal{M}\hn'$ (with $\,s\,$ rather than 
$\,r$), the coverings $\,\mathcal{M}_{\hh r\hn u}'\nh\to\mathcal{M}\hn'$ and 
$\,\mathcal{M}_{\hh su}'\nh\to\mathcal{M}\hn'$ have the same multiplicity, 
which completes the proof.
\end{proof}
\begin{remark}\label{delta}Replacing $\,\rd\,$ of (\ref{svp}) with $\,1/4\,$ 
times its original value, we can also require it to have the following 
property: {\it if\/ $\,\gamma\in\bg\,$ and\/ $\hs\,x\in\mathcal{E}\hs$ 
are such that both\/ $\,x$ and\/ $\,\gamma(x)$ lie in the\/ 
$\,\rd$-neigh\-bor\-hood of\/ $\,\mathcal{E}\nh'\nnh$, defined as in  
Sect.~{\rm\ref{as}}, then\/ $\,\gamma\in\stb\hn'$ for the stabilizer 
group\/ $\,\stb\hn'$ of\/ $\,\mathcal{E}\nh'$ given by\/ {\rm(\ref{stb})}}. 
In fact, letting $\,\mathcal{E}\nh''$ be the leaf of 
$\,F\hskip-3pt_{\mathcal{E}}\w$ through $\,x$, we see from (\ref{inv}) that 
its $\,\gamma$-im\-age $\,\gamma(\mathcal{E}\nh'')\,$ is also a leaf of 
$\,F\hskip-3pt_{\mathcal{E}}\w$, while both leaves are within the distance 
$\,\rd\,$ from $\,\mathcal{E}\nh'\nnh$, which yields 
$\,\mathrm{dist}\hh(\mathcal{E}\nh''\nnh,\gamma(\mathcal{E}\nh''))<2\rd\,$ and 
so, due to the triangle inequality, 
$\,\mathrm{dist}\hh(\mathcal{E}\nh'\nnh,\gamma(\mathcal{E}\nh'))
\le\mathrm{dist}\hh(\mathcal{E}\nh'\nnh,\mathcal{E}\nh'')
+\mathrm{dist}\hh(\mathcal{E}\nh''\nnh,\gamma(\mathcal{E}\nh''))
+\mathrm{dist}\hh(\gamma(\mathcal{E}\nh''),\gamma(\mathcal{E}\nh'))
<\rd+2\rd+\rd=4\rd$. Thus, $\,x+ru\in\gamma(\mathcal{E}\nh')\,$ for some 
$\,x\in\mathcal{E}\nh'\nnh$, some unit vector $\,u\in\tvs$ orthogonal to 
$\,\tvs\hn'\nnh$, and 
$\,r=\mathrm{dist}\hh(\mathcal{E}\nh'\nnh,\gamma(\mathcal{E}\nh'))\in[\hs0,4\rd)$. 
Assuming now (\ref{svp}) with $\,\rd$ replaced by $\,4\rd$, one gets 
$\,r=0$, that is, $\,\gamma(\mathcal{E}\nh')=\mathcal{E}\nh'$ and 
$\,\gamma\in\stb\hn'\nnh$. Namely, the $\,\mathrm{pr}$-im\-age of the curve 
$\hs[\hs0,4\rd)\ni t\mapsto x+tu\hs$ is a geodesic in the image of the 
dif\-feo\-mor\-phism $\,\mathrm{Exp}^\perp$ of Lemma~\ref{nrexp}(a), which 
intersects $\,\mathcal{M}\hn'$ only at $\,t=0$, while 
$\,\mathcal{M}\hn'\nh=\mathrm{pr}(\mathcal{E}\nh')
=\mathrm{pr}(\gamma(\mathcal{E}\nh'))$, since 
$\,\mathcal{M}=\mathcal{E}/\hh\bg$.
\end{remark}
\begin{lemma}\label{cmmut}
Let there be given\/ $\,\tvs\hn'\nnh,\mathcal{E}\nh'$ as in 
Lemma\/~{\rm\ref{isolv}}, $\,\rd\,$ having the additional property of 
Remark\/~{\rm\ref{delta}}, any\/ $\,r\in(0,\rd)$, and any unit vector\/ 
$\,u\in\tvs$ orthogonal to\/ $\,\tvs\hn'\nnh$.
\begin{enumerate}
  \def\theenumi{{\rm\alph{enumi}}}
\item The stabilizer group\/ $\,\stb_{\hh r\hn u}'$ in\ 
{\rm(\ref{svp})} does not depend on\/ $\,r\in(0,\rd)$.
\item The linear part of each element of\/ $\,\stb_{\hh r\hn u}'$ 
keeps\/ $\,u\,$ fixed.
\item $\stb_{\hh r\hn u}'$ is a sub\-group of\/ $\,\stb_0'$ 
with the finite index\/ $\,k=k(u)\ge1\,$ of \,Lemma\/~{\rm\ref{isolv}},
\item $\mathrm{pr}:\mathcal{E}\to\mathcal{M}\,$ 
maps the\/ $\,\rd$-neigh\-bor\-hood\/ $\,\mathcal{E}\hskip-2pt_\rd\w$ of\/ 
$\,\mathcal{E}\nh'$ in\/ $\,\mathcal{E}\hs$ onto\/ $\,\mathcal{M}'\hskip-4.1pt_\rd\w\hs$
of \,Lemma\/~{\rm\ref{nrexp}(a)}.
\item $\mathcal{E}\hskip-2pt_\rd\w$ \ and\/
$\,\mathcal{M}'\hskip-4.1pt_\rd\w\,$ are
unions of leaves of, respectively, $\,F\hskip-3pt_{\mathcal{E}}\w\,$ and\/ 
$\,F\hskip-3.8pt_{\mathcal{M}}\w$.
\item The pre\-im\-age under\/ 
$\,\mathrm{pr}:\mathcal{E}\hskip-2pt_\rd\w\nh\to\mathcal{M}\nh_\rd\w\hskip2pt$
of the leaf\/ $\,\mathcal{M}_{\hh r\hn u}'=\hs\mathrm{pr}(\mathcal{E}\nh'\nnh+ru)\,$
of\/ $\,F\hskip-3.8pt_{\mathcal{M}}\w\,$ equals the union of the images\/ 
$\,\gamma(\mathcal{E}\nh'\nnh+ru)\,$ over all\/ $\,\gamma\in\stb_0'$.
\end{enumerate}
\end{lemma}
\begin{proof}By (\ref{cpr}) and (\ref{svp}), 
$\,\mathcal{M}_v'=(\mathcal{E}\nh'\nnh+v)/\hs\bg_{\nh v}'$, if one lets 
$\,\bg_{\nh v}'$ denote the image of $\,\stb_{\nh v}'$ under the 
injective homo\-mor\-phism of restriction to $\,\mathcal{E}\nh'\nnh$, cf.\ 
Theorem~\ref{restr}(iii). Fixing $\,s\in[\hs0,\rd)$ and $\,r\in(0,\rd)\,$ we 
therefore conclude from Lemma~\ref{isolv} and (\ref{iff}) that, whenever 
$\,x\in\mathcal{E}\nh'\nnh+ru$ and $\,\gamma\in\stb_{\hh r\hn u}'$, 
there exists $\,\hat\gamma\in\stb_{\hn su}'$ satisfying the condition
\begin{equation}\label{gxp}
\gamma(x)\,+\,v\,=\,\hat\gamma(x+v)\hh,\hskip12pt\mathrm{where\ 
}\,\,v=(s-r)u\hh,\hskip8pt\mathrm{and\ }\,\,\hat\gamma=\gamma\,\,\mathrm{\ 
when\ }\,\,s=r,
\end{equation}
the last clause being obvious since $\,\gamma,\hat\gamma\in\bg\,$ and the 
action of $\,\bg\,$ is free. With $\,u\,$ and $\,\gamma\,$ fixed as well, for 
each given $\,\hat\gamma\in\stb_{\hn su}'$ the set of all 
$\,x\in\mathcal{E}\nh'\nnh+ru\,$ having Property (\ref{gxp}) is closed in 
$\,\mathcal{E}\nh'\nnh+ru\,$ while, as we just saw, the union of these sets over 
all $\,\hat\gamma\in\stb_{\hn su}'$ equals $\,\mathcal{E}\nh'\nnh+ru$. Thus, 
by Baire's theorem (Lemma~\ref{baire}), some 
$\,\hat\gamma\in\stb_{\hn su}'$ 
satisfies (\ref{gxp}) with all $\,x\,$ from some nonempty open subset of 
$\,\mathcal{E}\nh'\nnh+ru$, and hence -- by real-an\-a\-lyt\-ic\-i\-ty -- for 
all $\,x\in\mathcal{E}\nh'\nnh+ru$. In terms of the translation $\,\tau\nnh_v\w$ 
by the vector $\,v$, we consequently have 
$\,\hat\gamma=\tau\nnh_v\w\circ\gamma\circ\tau\nnh_v^{-\nh1}$ on 
$\,\mathcal{E}\nh'\nnh+su$, so that, due to Theorem~\ref{restr}(iii),
$\,\gamma\,$ uniquely determines $\,\hat\gamma$, while the 
assignment $\,\gamma\mapsto\hat\gamma\,$ is a homo\-mor\-phism 
$\,\stb_{\hh r\hn u}'\to\stb_{\hn su}'\nh\subseteq\hh\bg\hh$, and 
$\,\zeta
=\hat\gamma\circ\tau\nnh_v\w\circ\gamma^{-\nh1}\nh\circ\tau\nnh_v^{-\nh1}$ 
equals the identity on $\,\mathcal{E}\nh'\nnh+su$. If we now allow $\,s\,$ to 
vary from $\,r$ to $\,0$, the resulting curve $\,s\mapsto\zeta$ consists of
\af\ extensions, defined as in (\ref{afx}), of linear 
self-isom\-e\-tries of the orthogonal complement of $\,\tvs\hn'\nnh$, and 
$\,\hat\gamma=\zeta\circ\tau\nnh_v\w\circ\gamma\circ\tau\nnh_v^{-\nh1}$ on 
$\,\mathcal{E}$. As $\,\bg\,$ is discrete, the curve 
$\,s\mapsto\hat\gamma\in\bg\hh$, with $\,v=(s-r)u$, must be constant, and can 
be evaluated by setting $\,s=r\,$ (or, $\,v=0$). Thus, 
$\,\hat\gamma=\gamma\,$ on $\,\mathcal{E}\hs$ from the last clause of 
(\ref{gxp}), and so 
$\,\stb_{\hh r\hn u}'\nh\subseteq\hh\stb_{\hn su}'$. For $\,s>0$, 
switching $\,r$ with $\,s\,$ we get the opposite inclusion, and (a) follows. 
Also, taking the linear part of the resulting relation 
$\,\gamma=\zeta\circ\tau\nnh_v\w\circ\gamma\circ\tau\nnh_v^{-\nh1}$, we see 
that $\,\zeta\,$ equals the identity, for all $\,s$. Hence 
$\,\gamma=\tau\nnh_v\w\circ\gamma\circ\tau\nnh_v^{-\nh1}\nh$, that is,
$\,\gamma\,$ commutes with 
$\,\tau\nnh_v\w$ which, by (\ref{cnj}), amounts to (b). Setting 
$\,s=0$, we obtain the first part of (c): 
$\,\stb_{\hh r\hn u}'\nh\subseteq\hh\stb_0'$. Assertion (d) is 
clear as $\,\mathrm{pr}$, being locally isometric, maps line segments onto 
geodesic segments. Lemma~\ref{fltmf}(a) for 
$\,D=F\hskip-3.8pt_{\mathcal{M}}\w$ gives (e). With 
$\,\mathrm{pr}:\mathcal{E}\to\mathcal{M}=\mathcal{E}/\hh\bg\,$ in
Theorem~\ref{restr}(i), the additional property of $\,\rd\hs$
(Remark~\ref{delta}) yields (f). Finally, 
for $\,k=k(u)$, the geodesic $\,[\hs0,r]\ni t\mapsto\mathrm{pr}(x+tu)$, normal 
to $\,\mathcal{M}\hn'$ at $\,y=\mathrm{pr}(x)$, is one of $\,k\,$ such 
geodesics $\,[\hs0,r]\ni t\mapsto\mathrm{pr}(x+tw)$, joining $\,y\,$ to points 
of its pre\-im\-age under the 
projection $\,\mathcal{M}_{\hh r\hn u}'\nh\to\mathcal{M}\hn'$ of 
Lemma~\ref{isolv}, where $\,w\,$ ranges over a $\,k$-el\-e\-ment set 
$\,\mathcal{R}\,$ of unit vectors in $\,\tvs\nh$, orthogonal to 
$\,\tvs\hn'\nnh$. The union of the corresponding subset 
$\,C=\{\mathcal{E}\nh'\nnh+rw:w\in\mathcal{R}\}$ of the leaf space of 
$\,F\hskip-3pt_{\mathcal{E}}\w$ equals the pre\-im\-age in (f) -- and hence an 
orbit for the action of $\,\stb_0'$ -- as every leaf in the 
pre\-im\-age contains a point nearest $\,x$. Due to the 
al\-ready-es\-tab\-lished inclusion 
$\,\stb_{\hh r\hn u}'\nh\subseteq\hh\stb_0'$ and (\ref{svp}), 
$\,\stb_{\hh r\hn u}'$ is the iso\-tropy group of 
$\,\mathcal{E}\nh'\nnh+ru$ relative to the transitive action of 
$\,\stb_0'$ on $\,C$, and so $\,k$, the cardinality of $\,C$, equals 
the index of $\,\stb_{\hh r\hn u}'$ in $\,\stb_0'$, which 
proves the second part of (c).
\end{proof}

\section{The generic stabilizer group}\label{gg}
Given a Bie\-ber\-bach group $\,\bg\,$ in a Euclidean \af\ space 
$\,\mathcal{E}\hs$ with the translation vector space $\,\tvs\nnh$, let us fix 
a vector subspace $\,\tvs\hn'$ of $\,\tvs\hs$ satisfying (\ref{lrd}) for
$\,L,H\,$ introduced in (\ref{exa}). As long 
as $\,\dim\,\mathcal{E}\ge2$, such $\,\tvs\hn'$ always exists 
(Sect.~\ref{lr}). An element $\,\mathcal{E}\nh'$ of 
$\,\mathcal{E}\hn/\hn\tvs\hn'\nnh$, that is, a coset of $\,\tvs\hn'$ in 
$\,\mathcal{E}$, will be called {\it generic\/} if its stabilizer
group $\,\stb\hn'\subseteq\bg\hh$, defined by (\ref{stb}), equals
\begin{equation}\label{krh}
\mathrm{the\ kernel\ of\ the\ homo\-mor\-phism\ 
}\,\bg\to\mathrm{Iso}\,[\mathcal{E}\hn/\hn\tvs\hn']\mathrm{\ corresponding\ 
to\ (\ref{act}).}
\end{equation}
The $\,\mathrm{pr}$-im\-ages of generic cosets 
of $\,\tvs\hn'$ will be called generic leaves 
of $\,F\hskip-3.8pt_{\mathcal{M}}\w$.

Still using the symbols $\,L,H\,$ and $\,\mathrm{pr}\,$ for
the groups appearing in (\ref{lah}) -- (\ref{exa}) and 
the u\-ni\-ver\-sal-cov\-er\-ing projection 
$\,\mathcal{E}\to\mathcal{M}=\mathcal{E}/\hh\bg\hh$, let us also
\begin{equation}\label{kpr}
\begin{array}{l}
\mathrm{denote\ by\ }\hs K'\hn\nnh\subseteq\hn H\hs\mathrm{\ the\ normal\
sub\-group\ 
consisting\ of\ all\ elements\ of}\\
H\,\mathrm{\ that\ act\ on\ the\ orthogonal\ complement\ 
of\ }\,\tvs\hn'\hn\mathrm{\ as\ the\ identity,\hskip-1.1pt\ and}\\
\mathrm{by\ }\,\hs\mathcal{U}'\hs\mathrm{\ the\ subset\ of\ 
}\,\mathcal{E}\hn/\hn\tvs\hn'\mathrm{\ formed\ by\ all\ generic\ cosets\ 
of\ }\,\tvs\hn'\mathrm{\ in\ }\,\mathcal{E}.
\end{array} 
\end{equation}
\begin{theorem}\label{gnric}
For\/ $\,\stb\hn'$ equal to\/ {\rm(\ref{krh})}, under the assumptions
preceding\/ {\rm(\ref{krh})}, with the notation of\/ {\rm(\ref{kpr})}, one
has the following conclusions.
\begin{enumerate}
  \def\theenumi{{\rm\roman{enumi}}}
\item $\mathcal{U}'$ in\/ {\rm(\ref{kpr})} constitutes an open dense subset
of\/ 
$\,\mathcal{E}\hn/\hn\tvs\hn'\nnh$.
\item The normal sub\-group\/ $\,\stb\hn'$ of\/ $\hs\bg\,$ is contained as a 
fi\-nite-in\-dex sub\-group in the stabilizer group of every\/ 
$\,\mathcal{E}\nh'\nnh\in\mathcal{E}\hn/\hn\tvs\hn'$ for the action\/ 
{\rm(\ref{act})}, and equal to this stabilizer group if\/ 
$\,\mathcal{E}\nh'\nnh\in\mathcal{U}'\nnh$.
\item The\/ $\,\mathrm{pr}$-im\-ages\/ 
$\,\mathcal{M}\hn'\nnh,\mathcal{M}\hn''$ of any\/ 
$\,\mathcal{E}\nh'\nnh,\mathcal{E}\nh''\nnh\in\mathcal{U}'$ are isometric to each 
other.
\item If one identifies\/ $\,\mathcal{E}\hs$ with its translation vector space 
$\,\tvs$ via a choice of an origin, $\,\stb\hn'\nh$ becomes the set of all the 
elements of\/ $\,\bg\,$ having, for\/ $\,K\hn'$ given by\/ {\rm(\ref{kpr})}, 
the form
\begin{equation}\label{axb}
\tvs\nh\ni\hn x\hn\mapsto\nh Ax\nh+\hn 
b\in\tvs\nnh\mathrm{,\ in\ which\ }\,b\nh\in\nnh\tvs\hn'\nnh\mathrm{\ and\ 
the\ linear\ part\ }\hs A\,\mathrm{\ lies\ in\ }\hs K\hn'\nnh.
\end{equation}
\item Whenever\/ $\,\mathcal{E}\nh'\nnh\in\mathcal{U}'\nnh$, the homo\-mor\-phism 
which restricts elements of the generic stabilizer group\/ $\,\stb\hn'\nh$ 
to\/ $\,\mathcal{E}\nh'$ is injective, and the resulting iso\-mor\-phic image\/ 
$\,\bg\hn'$ of\/ $\,\stb\hn'\nh$ constitutes a Bie\-ber\-bach group in the 
Euclidean \af\ space\/ $\,\mathcal{E}\nh'\nnh$. The lattice sub\-group of\/ 
$\,\bg\hn'$ and its holonomy group\/ $\,H\nh'$ are the intersection\/ 
$\,L\nh'\nh=L\cap\tvs\hn'$ and the image\/ $\,H\nh'$ of the group\/ $\,K\hn'$ 
defined in\/ {\rm(\ref{kpr})} under the injective homo\-mor\-phism of 
restriction to\/ $\,\tvs\hn'\nnh$.
\end{enumerate}
\end{theorem}
\begin{proof}Lemma~\ref{cmmut}(a) states that the assumptions of 
Lemma~\ref{opdns} are satisfied by the Euclidean \af\ space 
$\,\mathcal{W}\nnh=\mathcal{E}\hn/\hn\tvs\hn'$ and the mapping $\,F\,$ that 
sends $\,\mathcal{E}\nh'\nnh\in\mathcal{E}\hn/\hn\tvs\hn'$ to its stabilizer 
group $\,\stb\hn'$ with (\ref{stb}). The assignment 
$\,\mathcal{E}\nh'\nh\mapsto\stb\hn'$ is thus {\it locally constant\/} on
some open dense set
$\,\mathcal{U}'\nh\subseteq\hs\mathcal{E}\hn/\hn\tvs\hn'\nnh$. Letting 
$\,\stb\hn'$ be the constant value of this assignment assumed on a nonempty 
connected open subset $\,\mathcal{W}'$ of $\,\mathcal{U}'\nnh$, and fixing 
$\,\gamma\in\stb\hn'\nnh$, we obtain $\,\gamma(\mathcal{E}\nh')=\mathcal{E}\nh'$ for 
all $\,\mathcal{E}\nh'\nnh\in\mathcal{W}'\nnh$, and hence, from 
real-an\-a\-lyt\-ic\-i\-ty, for all 
$\,\mathcal{E}\nh'\nnh\in\mathcal{E}\hn/\hn\tvs\hn'\nnh$. Thus, our 
$\,\stb\hn'$ is contained in the stabilizer group of every 
$\,\mathcal{E}\nh'\nnh\in\mathcal{E}\hn/\hn\tvs\hn'\nnh$. Since the same applies 
also to another constant value $\,\stb\hn''$ assumed on a nonempty connected 
open set, $\,\stb\hn''\nnh=\stb\hn'$ and the phrase `locally constant' may be 
replaced with {\it constant}. By Lemma~\ref{cmmut}(c), any such
$\,\stb\hn'$ 
must be a fi\-nite-in\-dex sub\-group of each stabilizer group. As
$\,\stb\hn'$ consists of the elements of $\,\bg\,$ preserving every 
$\,\mathcal{E}\nh'\nnh\in\mathcal{U}'\nnh$, real-an\-a\-lyt\-ic\-i\-ty implies 
that they preserve all $\,\mathcal{E}\nh'\nnh\in\mathcal{E}\hn/\hn\tvs\hn'\nnh$, 
and so $\,\stb\hn'$ coincides with (\ref{krh}), which also shows that 
$\,\stb\hn'$ is a normal sub\-group of $\,\bg\hh$, and (i) -- (ii) follow.

Next, (iv) holds since the set of all 
$\,\gamma\in\mathrm{Aut}\,\mathcal{E}\hs$ satisfying (\ref{cnd}) 
clearly constitutes a sub\-group of $\,\mathrm{Aut}\,\mathcal{E}\hs$ 
containing, as a normal sub\-group, the set of all $\,\gamma$ with 
(\ref{cnd}) such that trans\-la\-tion\-al-part coset of $\,\gamma\,$ 
(see Definition~\ref{trprt}) is contained in $\,\tvs\hn'\nnh$. To verify this 
claim, note that the latter set is a normal sub\-group, being the kernel of
the
obvious homo\-mor\-phism from the sub\-group of all $\,\gamma\,$ having the 
property (\ref{cnd}) into 
the translation sub\-group $\,\tvs\nnh/\hn\tvs\hn'$ of
$\,\mathrm{Aut}\,[\mathcal{E}\hn/\hn\tvs\hn']$. More 
precisely, $\,\gamma\,$ represented by the pair $\,(A,b)\,$ (as in  
Definition~\ref{trprt}) preserves each element of
$\,\mathcal{E}\hn/\hn\tvs\hn'$ 
if and only if $\,Av\nh+\hn b\,$ differs from $\,v$, for every 
$\,v\in\tvs\nh$, by an element of $\,\tvs\hn'$ or, equivalently (as one sees 
setting $\,v=0$), $\,\tvs\hn'$ contains both $\,b\,$ and
$\,(A-\hn1)(\tvs)$.

Finally, Theorem~\ref{restr}(ii)\hs-\hs(iii) yields (v), 
while (v) implies (iii) via Remark~\ref{bijct}.
\end{proof}
The example provided by a compact flat manifold
$\,\mathcal{M}=\mathcal{E}/\hh\bg\,$ which is a Riemannian product 
$\,\mathcal{M}\nh=\nnh\mathcal{M}\hn'\nh\times\hn\mathcal{M}\hn''$ 
with $\,\mathcal{E}=\mathcal{E}\nh'\nnh\times\mathcal{E}\nh''$ and 
$\,\bg=\bg\hn'\nnh\times\bg\hn''$ for two Bie\-ber\-bach groups 
$\,\bg\hn'\nh,\bg\hn''$ in Euclidean \af\ spaces 
$\,\mathcal{E}\nh'\nh,\mathcal{E}\nh''$ having the translation vector spaces 
$\,\tvs\hn'\nnh,\tvs\hn''\nnh$, while $\,\mathcal{M}\hn'$ is not a torus,
shows that, in general,
\begin{equation}\label{nni}
\mathrm{an\ element\ of\ }\,\bg\hs\mathrm{\ acting\ trivially\ on\
}\,\mathcal{E}\hn/\hn\tvs\hn'\mathrm{\ need\ not\ lie\ in\
}\,L\nh'\nh.
\end{equation}
Namely, the $\,H\nh$-in\-var\-i\-ant sub\-space
$\,\tvs\hn'\nh\times\nh\{0\}\,$ then gives rise to the $\,\mathcal{M}\hn'$
factor foliation $\,F\hskip-3pt_{\mathcal{E}}\w$ of the product manifold
$\,\mathcal{M}$, and the action of the group $\,\bg\hn'\nh\times\nh\{1\}\,$ on
its leaf space is obviously trivial, even though not all elements of
$\,\bg\hn'\nh\times\nh\{1\}\,$ are translations.
\begin{lemma}\label{subgp}
Using any given\/ $\,\mathcal{E}\nh'\nnh\in\mathcal{U}'$ in
Theorem\/~{\rm\ref{gnric}}, where\/ $\,\tvs\hn'$ satisfying\/ {\rm(\ref{lrd})} 
is fixed, let us identify\/ $\,\stb\hn'$ with\/ $\,\bg\hn'$
and\/ $\,H\nh'$ with\/ $\,K\hn'$ via the iso\-mor\-phisms\/
$\,\stb\hn'\nh\to\bg\hn'$ and\/ $\,K\hn'\nh\to H\nh'$ 
resulting from Theorem\/{\rm~\ref{gnric}(v)}, which turns\/
$\,\bg\hn'$ and\/ $\,H\nh'$ into sub\-groups of\/ $\,\bg\,$ and\/ 
$\,H$. These sub\-groups\/ $\,\bg\hn'\nh\subseteq\bg\,$ and\/
$\,H\nh'\nh\subseteq H\,$ do not depend on the choice of\/
$\,\mathcal{E}\nh'\nnh\in\mathcal{U}'\nnh$, and neither does the mapping
degree\/ $\,d=|H\nh'|\,$ of the $\,d\hh$-fold covering 
projection $\,\mathcal{T}\hh'\nnh\to\nh\mathcal{M}\hn'\nh
=\mathcal{T}\hh'\nnh/\nh H\nh'$ analogous to those mentioned in\/
{\rm(\ref{cmp})}.
\end{lemma}
\begin{proof}Our claims about $\,\bg\hn'$ and $\,H\nh'$ are immediate
from (\ref{krh}) and (\ref{kpr}).
\end{proof}
Any lattice $\,L\,$ in the translation vector space 
$\,\tvs\hs$ of a Euclidean \af\ space $\,\mathcal{E}\hs$ is, obviously, 
a Bie\-ber\-bach group in $\,\mathcal{E}$. In the case of a fixed vector 
subspace $\,\tvs\hn'$ of $\,\tvs$ with (\ref{lrd}), all the general facts 
established about any given Bie\-ber\-bach group $\,\bg\,$ in 
$\,\mathcal{E}$, the compact flat manifold 
$\,\mathcal{M}=\mathcal{E}/\hh\bg\hh$, and the leaves $\,\mathcal{M}\hn'$ of 
$\,F\hskip-3.8pt_{\mathcal{M}}\w\hs$ (see Theorem~\ref{restr}) thus remain 
valid for the torus $\,\mathcal{T}=\,\mathcal{E}/\nh L\,$ and the leaves 
$\,\mathcal{T}\hh'$ of $\,F\hskip-3pt_{\mathcal{T}}\w$. Every coset of 
$\,\tvs\hn'$ is generic if we declare the lattice $\,L\,$ of $\,\bg\,$ 
to be the new Bie\-ber\-bach group.

\section{The leaf space}\label{ls}
We again adopt the assumptions and notation of Theorem~\ref{restr}. 
Not surprisingly, the leaf space
$\,\mathcal{M}/\nh F\hskip-3.8pt_{\mathcal{M}}\w$ has the following
property (discussed below):
$\,\mathcal{M}/\nh F\hskip-3.8pt_{\mathcal{M}}\w$ forms
\begin{equation}\label{lsp}
\begin{array}{l}
\mathrm{a\ flat\ 
compact\ or\-bi\-fold,\nnh\ canonically\ identified\ with\ the\ quotient\ of\
the}\\
\mathrm{torus\ }[\mathcal{E}\hn/\hn\tvs\hn']/[L\cap\tvs\hn']\mathrm{\ under\ 
the\ isometric\ action\ of\ the\ finite\ group\ }H\nh.
\end{array}
\end{equation}
By a {\it crystallographic group\/} \cite{szczepanski} in a Euclidean 
\af\ space one means a discrete group of isometries having a compact 
fundamental domain, cf.\ Remark~\ref{cptfd}.
\begin{proposition}\label{cryst}
Under the assumptions listed in the first line of Sect.~{\rm\ref{gg}}, 
with\/ $\,\stb\hn'$ denoting the normal sub\-group\/ {\rm(\ref{krh})} of\/ 
$\,\bg\hh$, the quotient group\/ $\,\bg/\hn\stb\hn'$ acts effectively by 
isometries on the quotient Euclidean \af\ space\/ 
$\,\mathcal{E}\hn/\hn\tvs\hn'$ and, when identified with a sub\-group of\/ 
$\,\mathrm{Iso}\,[\mathcal{E}\hn/\hn\tvs\hn']$, it constitutes a 
crystallographic group.
\end{proposition}
\begin{proof}A compact fundamental domain exists since
$\,\bg/\hn\stb\hn'$ contains the lattice sub\-group $\,L/\nh L\nh'$ of
$\,\tvs\nh/\tvs\hn'$ (see the end of Sect.~\ref{lv}). To verify 
the discreteness of $\,\bg/\hn\stb\hn'\nnh$, suppose that, on the contrary,
some sequence $\,\gamma\hskip-1.7pt_k\w\in\bg\,$, $\,k=1,2\,\dots\hs$, has terms
representing 
mutually distinct elements of $\,\bg/\hn\stb\hn'$ which converge in 
$\,\mathrm{Iso}\,[\mathcal{E}\hn/\hn\tvs\hn']$. As $\,L\nh'$ is a 
lattice in $\,\tvs\hn'\nnh$, fixing $\,x\in\mathcal{E}\hs$ and suitably 
choosing $\,v\nh_k\w\in L\nh'$ we achieve boundedness of the sequence 
$\,\hat\gamma\hskip-1.7pt_k\w(x)=\gamma\hskip-1.7pt_k\w(x)+v\nh_k\w$, while
$\,\hat\gamma\hskip-1.7pt_k\w$ represent 
the same (distinct) elements of $\,\bg/\hn\stb\hn'$ as
$\,\gamma\hskip-1.7pt_k\w$. The 
ensuing  convergence of a sub\-se\-quence of $\,\hat\gamma\hskip-1.7pt_k\w$ contradicts 
the discreteness of $\,\bg$.
\end{proof}
The resulting quotient of $\,\mathcal{E}\hn/\hn\tvs\hn'$ under the action of 
$\,\bg/\hn\stb\hn'$ is thus a flat compact or\-bi\-fold \cite{davis}, which 
may clearly be identified both with the leaf space 
$\,\mathcal{M}/\nh F\hskip-3.8pt_{\mathcal{M}}\w$, and with the quotient of 
the torus $\,[\mathcal{E}\hn/\hn\tvs\hn']/[L\cap\tvs\hn']\,$
mentioned in (\ref{lsp}). The latter identification clearly implies 
the Haus\-dorff property of the leaf space
$\,\mathcal{M}/\nh F\hskip-3.8pt_{\mathcal{M}}\w$.

On the other hand, for an $\,H\nh$-in\-var\-i\-ant sub\-space $\,\tvs\hn''$ 
of $\,\tvs\hs$ {\it not\/} assumed to be an $\,L$-sub\-space, there 
exists an $\,L${\it-clo\-sure\/} of $\,\tvs\hn''\nnh$, meaning the 
smallest $\,L$-sub\-space $\,\tvs\hn'$ of $\,\tvs\hs$ containing 
$\,\tvs\hn''\nnh$, which is obviously obtained by intersecting all such 
$\,L$-sub\-spaces (Lemma~\ref{spint}). The leaf space 
$\,\mathcal{M}/\nh F\hskip-3.8pt_{\mathcal{M}}\w$ corresponding to 
$\,\tvs\hn'$ then forms a natural ``Haus\-dorff\-iz\-ation'' of the leaf space 
of $\,\tvs\hn''\nnh$, and may also be described in terms of 
Haus\-dorff-Gro\-mov limits. See the recent paper 
\cite{bettiol-derdzinski-mossa-piccione}.

\section{Intersections of generic complementary leaves}\label{ig}
Throughout this section $\,\bg\,$ is a given Bie\-ber\-bach group in a 
Euclidean \af\ space $\,\mathcal{E}\hs$ of dimension $\,n\ge2$, while 
$\,\tvs\hn'\nnh,\tvs\hn''$ are two mutually complementary 
$\,H\nnh$-in\-var\-i\-ant $\,L$-sub\-spaces of the translation vector space 
$\,\tvs\hs$ of $\,\mathcal{E}$, in the sense of (\ref{cpl}) and 
Definition~\ref{lsbsp}, for $\,L\,$ and $\,H\,$ associated with $\,\bg\,$ via 
(\ref{lah}). We also fix generic cosets $\,\mathcal{E}\nh'$ of $\,\tvs\hn'$ 
and $\,\mathcal{E}\nh''$ of $\,\tvs\hn''$ (see the beginning of 
Sect.~\ref{gg}), which leads to the analogs 
$\,L\nh'\nh\to\hs\bg\hn'\nh\to H\nh'$ and 
$\,L\nh''\nh\to\hs\bg\hn''\nh\to H\nh''$ of (\ref{exa}), described by 
Theorem~\ref{restr}(ii) and, $\,\mathcal{E}\nh'\nnh,\mathcal{E}\nh''$ being generic, 
Theorem~\ref{gnric}(v) yields $\,L\nh'\nh=L\cap\tvs\hn'$ and 
$\,L\nh''\nh=L\cap\tvs\hn''\nnh$. Furthermore, for these 
$\,\bg\nh,\bg\hn'\nnh,\bg\hn''\nnh,L,L\nh'\nnh,L\nh''\nnh,
H,H\nh'\nnh,H\nh''\nnh$,
\begin{equation}\label{cnc}
\mathrm{the\ conclusions \ of\ Lemma~\ref{prdsg}\ hold\ if\ we\ replace\
}\,G\,\mathrm{\ with\ }\,\bg,L\,\mathrm{\ or\ }\,H,
\end{equation}
since so do the assumptions of Lemma~\ref{prdsg}, provided that one uses 
Lemma~\ref{subgp} to treat $\,\bg\hn'$ and $\,\bg\hn''$ (or, $\,H\nh'$ and 
$\,H\nh''$) as sub\-groups of $\,\bg\,$ (or, respectively, $\,H$). In fact, 
(\ref{axb}) and the description of $\,K\hn'$ in (\ref{kpr}) show that all 
$\,A\in K\hn'$ (and, among them, the linear parts of all elements of 
$\,\stb\hn'\nh=\bg\hn'$) leave invariant both $\,\tvs\hn'$ and 
$\,\tvs\hn''\nnh$, and act via the identity on the latter. (We have the 
obvious isomorphic identifications of $\,\tvs\nnh/\hn\tvs\hn'$ with 
$\,\tvs\hn''$ on the one hand, and with the orthogonal complement of 
$\,\tvs\hn'$ in $\,\tvs\hs$ on the other, while such $\,A\,$ descend to the 
identity auto\-mor\-phism of $\,\tvs\nnh/\hn\tvs\hn'\nnh$.) The same is 
clearly the case if one switches the primed symbols with the double-prim\-ed 
ones, while elements of $\,\stb\hn'$ now commute with those of $\,\stb\hn''$ 
in view of (\ref{axb}). This yields (\ref{cnc}) and, consequently, allows us
to form
\begin{equation}\label{qgr}
\mathrm{the\ quotient\ groups\ }\,\hat\bg\nh
=\bg/(\bg\hn'\hh\bg\hn'')\hh,\hskip8pt\hat L
=L/(L\nh'\nnh L\nh'')\hh,\hskip8pt\hat H\nh=H/(H\nh'\nnh H\nh'')\hh.
\end{equation}
Finally, let $\,\mathrm{pr}:\mathcal{E}\to\mathcal{M}=\mathcal{E}/\hh\bg\,$ 
and 
$\,\mathcal{M}\hn'\nnh,\mathcal{M}\hn''\nnh,\mathcal{T}\hh'\nnh,
\mathcal{T}\hh''$ denote, respectively, the covering projection of 
Theorem~\ref{restr}(i), the $\,\mathrm{pr}$-im\-age $\,\mathcal{M}\hn'$ of 
$\,\mathcal{E}\nh'$ (or, $\,\mathcal{M}\hn''$ of $\,\mathcal{E}\nh''$), and the tori 
$\,\mathcal{E}\nh'\nnh/\nh L\nh'$ and $\,\mathcal{E}\nh''\nnh/\nh L\nh''\nnh$, 
contained in the torus $\,\mathcal{T}\nh=\hs\mathcal{E}/\nh L\,$ of 
(\ref{cmp}). Note that $\,\mathcal{M}\hn'$ and $\,\mathcal{M}\hn''$ are 
(compact) leaves of the parallel distributions arising, due to 
Theorem~\ref{restr}(i), on $\,\mathcal{M}=\mathcal{E}/\hh\bg\hh$, which 
itself is a compact flat Riemannian manifold (Remark~\ref{frpdc}).

For a ho\-mol\-o\-gy interpretation of parts (a) and (c) below, see
Theorem~\ref{orien}.
\begin{theorem}\label{intrs}
Under the above hypotheses, the following conclusions hold.
\begin{enumerate}
  \def\theenumi{{\rm\alph{enumi}}}
\item $\mathcal{M}\hn'\cap\mathcal{M}\hn''\nnh$, or\/ 
$\,\mathcal{T}\hh'\nnh\cap\mathcal{T}\hh''\nnh$, is a finite subset of\/ 
$\,\mathcal{M}$, or $\,\mathcal{T}\nnh$, and stands in a bijective 
correspondence with the quotient group\/ $\,\hat \bg\,$ or, 
respectively, $\,\hat L$, of\/ {\rm(\ref{qgr})},
\item The projection\/ $\,\mathcal{T}\nnh\to\mathcal{M}\,$ in\/ 
  {\rm(\ref{cmp})} maps\/ $\,\mathcal{T}\hh'\nnh\cap\mathcal{T}\hh''$ 
injectively into\/ $\,\mathcal{M}\hn'\cap\mathcal{M}\hn''\nnh$.
\item The cardinality\/ $\,|\mathcal{M}\hn'\cap\mathcal{M}\hn''\hn|\,$ of\/ 
$\,\mathcal{M}\hn'\cap\mathcal{M}\hn''$ equals\/ 
$\,|\mathcal{T}\hh'\nnh\cap\mathcal{T}\hh''\hn|\,$ times\/ $\,|\hat H|$.
\item The claim about\/ $\,\mathcal{T}\hh'\nnh\cap\mathcal{T}\hh''$ in\/
{\rm(a)} remains true whether or not\/
$\,\mathcal{E}\nh'\nnh,\mathcal{E}\nh''$ are generic.
\item The two bijective correspondences in\/ {\rm(a)} may be chosen so that, 
under the resulting identifications, the injective mapping\/ 
$\,\mathrm{pr}:\mathcal{T}\hh'\nnh\cap\mathcal{T}\hh''\nh
\to\mathcal{M}\hn'\cap\mathcal{M}\hn''$ of\/ {\rm(b)} coincides with the 
group homo\-mor\-phism\/ 
$\,\hat L\to\hat \bg\,$ induced by the inclusion\/ $\,L\to\bg$.
\end{enumerate}
\end{theorem}
\begin{proof}We first prove (a) for 
$\,\mathcal{M}\hn'\cap\mathcal{M}\hn''\nnh$. The finiteness of 
$\,\mathcal{M}\hn'\cap\mathcal{M}\hn''$ follows as 
$\,\mathcal{M}\hn'\nnh,\mathcal{M}\hn''\nnh$, and hence also 
$\,\mathcal{M}\hn'\cap\mathcal{M}\hn''\nnh$, are compact totally geodesic 
sub\-man\-i\-folds of $\,\mathcal{M}$, while 
$\,\mathcal{M}\hn'\cap\mathcal{M}\hn''\nnh$, nonempty by (\ref{opt}), has 
$\,\dim(\mathcal{M}\hn'\cap\mathcal{M}\hn'')=0\,$ due to (\ref{cpl}). The 
mapping $\,\varPsi:\bg\to\mathcal{M}\hn'\cap\mathcal{M}\hn''$ with 
$\,\mathrm{pr}(\mathcal{E}\nh'\nh\cap\gamma(\mathcal{E}\nh''))
=\{\varPsi(\gamma)\}\,$ is 
well defined in view of (\ref{opt}) applied to $\,\gamma(\mathcal{E}\nh'')\,$ 
rather than $\,\mathcal{E}\nh''\nnh$, and clearly takes values in both 
$\,\mathcal{M}\hn'\nh=\mathrm{pr}(\mathcal{E}\nh')$ and 
$\,\mathcal{M}\hn''\nh=\mathrm{pr}(\mathcal{E}\nh'')
=\mathrm{pr}(\gamma(\mathcal{E}\nh''))$. Surjectivity of $\,\varPsi\hs$ follows: 
if $\,\mathrm{pr}(x'')\in\mathcal{M}\hn'\cap\mathcal{M}\hn''\nnh$, where 
$\,x''\nh\in\mathcal{E}\nh''$ then, obviously, 
$\,\mathrm{pr}(x'')=\mathrm{pr}(x')\,$ and $\,x'\nh=\gamma(x'')\,$ for some 
$\,x'\nh\in\mathcal{E}\nh'$ and $\,\gamma\in\bg\hh$, so that 
$\,x'\nh\in\mathcal{E}\nh'\nh\cap\gamma(\mathcal{E}\nh'')\,$ and 
$\,\mathrm{pr}(x'')=\mathrm{pr}(x')\,$ equals $\,\varPsi(\gamma)$, the unique 
element of $\,\mathrm{pr}(\mathcal{E}\nh'\nh\cap\gamma(\mathcal{E}\nh''))$. 
Furthermore, $\,\varPsi$-pre\-im\-ages of elements of 
$\,\mathcal{M}\hn'\cap\mathcal{M}\hn''$ are precisely the cosets of the normal 
sub\-group $\,\bg\hn'\hh\bg\hn''\nnh$ of $\,\bg\,$ (which clearly implies (a) 
for $\,\mathcal{M}\hn'\cap\mathcal{M}\hn''$). Namely, the left and right 
cosets coincide, and so elements $\,\gamma\nh_1\w,\gamma\nh_2\w$ of 
$\,\bg\,$ lie in the same coset of $\,\bg\hn'\hh\bg\hn''$ if and only if 
\begin{equation}\label{smc}
\gamma'\nh\circ\gamma\nh_1\w=\gamma\nh_2\w\circ\gamma''\mathrm{\ \ for\ 
some\ }\,\gamma'\nh\in\bg\hn'\mathrm{\ and\ }\,\gamma''\nh\in\bg\hn''\nh.
\end{equation}
Now let $\,\gamma\nh_1\w,\gamma\nh_2\w$ lie in the same coset of 
$\,\bg\hn'\hh\bg\hn''\nnh$. For $\,\gamma'\nh,\gamma''$ with (\ref{smc}), 
$\,\gamma'(\mathcal{E}\nh')=\mathcal{E}\nh'$ and 
$\,\gamma''(\mathcal{E}\nh'')=\mathcal{E}\nh''$ by the definition (\ref{stb}) of 
$\,\stb\hn'\nnh,\stb\hn''$ and their identification with 
$\,\bg\hn'\nnh,\bg\hn''$ (see above). Thus, 
$\,\{\varPsi(\gamma\nh_1\w)\}
=\mathrm{pr}(\mathcal{E}\nh'\nh\cap\gamma\nh_1\w(\mathcal{E}\nh''))
=\mathrm{pr}(\gamma'(\mathcal{E}\nh'\nh\cap\gamma\nh_1\w(\mathcal{E}\nh'')))
=\mathrm{pr}(\gamma'(\mathcal{E}\nh')\cap\gamma'(\gamma\nh_1\w(\mathcal{E}\nh'')))
=\mathrm{pr}(\mathcal{E}\nh'\nh\cap\gamma'(\gamma\nh_1\w(\mathcal{E}\nh'')))
=\mathrm{pr}(\mathcal{E}\nh'\nh\cap\gamma\nh_2\w(\gamma''(\mathcal{E}\nh'')))$, and so $\,\{\varPsi(\gamma\nh_1\w)\}
=\mathrm{pr}(\mathcal{E}\nh'\nh\cap\gamma\nh_2\w(\mathcal{E}\nh''))
=\{\varPsi(\gamma\nh_2\w)\}$. Conversely, if 
$\,\gamma\nh_1\w,\gamma\nh_2\w\in\bg\,$ and 
$\,\varPsi(\gamma\nh_1\w)=\varPsi(\gamma\nh_2\w)$, the unique points 
$\,x\hn_1\w$ of $\,\mathcal{E}\nh'\nh\cap\gamma\nh_1\w(\mathcal{E}\nh'')\,$ 
and $\,x\hn_2\w$ of $\,\mathcal{E}\nh'\nh\cap\gamma\nh_2\w(\mathcal{E}\nh'')$ 
both lie in the same $\,\bg$-or\-bit, and hence 
$\,x\hn_2\w=\gamma(x\hn_1\w)\,$ with some $\,\gamma\in\bg$. For 
$\,\gamma'\nh=\gamma$ and 
$\,\gamma''\nh=\gamma\nh_2^{-\nh1}\nh\circ\gamma\circ\gamma\nh_1\w$, the image 
$\,\gamma'(\mathcal{E}\nh')\,$ (or, $\,\gamma''(\mathcal{E}\nh'')$) intersects 
$\,\mathcal{E}\nh'$ (or, $\,\mathcal{E}\nh''$), the common point being 
$\,x\hn_2\w=\gamma(x\hn_1\w)$ or, respectively, 
$\,\gamma\nh_2^{-\nh1}(x\hn_2\w)=\gamma\nh_2^{-\nh1}(\gamma(x\hn_1\w))$. From 
(\ref{iff}) we thus obtain $\,\gamma'\nh\in\stb\hn'\nh=\bg\hn'$ and 
$\,\gamma''\nh\in\stb\hn''\nh=\bg\hn''\nnh$, which yields (\ref{smc}).

Now (a) for $\,\mathcal{T}\hh'\nnh\cap\mathcal{T}\hh''\nnh$, and (d), 
follow as special cases; see the end of Sect.~\ref{gg}.

Except for the word `injective' the claim made in (e) is immediate if one uses 
the mapping $\,\varPsi:\bg\to\mathcal{M}\hn'\cap\mathcal{M}\hn''$ defined 
above and its analog $\,L\to\mathcal{T}\hh'\nnh\cap\mathcal{T}\hh''$ obtained by 
replacing $\,\bg\nh,\mathcal{M}\hn'\nnh,\mathcal{M}\hn''$ and 
$\,\mathrm{pr}\,$ with $\,L,\mathcal{T}\hh'\nnh,\mathcal{T}\hh''$ and the 
projection $\,\mathcal{E}\to\mathcal{T}\nh=\hs\mathcal{E}/\nh L$. This yields 
(b), injectivity of the homo\-mor\-phism 
$\,\hat L\to\hat \bg\,$ being immediate: if an element of $\,L\,$ lies in 
$\,\bg\hn'\hh\bg\hn''$ (and hence has the form 
$\,\gamma'\nh\circ\gamma''\nnh$, where 
$\,(\gamma'\nnh,\gamma'')\in\bg\hn'\nh\times\bg\hn''$), (\ref{axb}) implies 
that $\,\gamma'\nnh,\gamma''$ are translations with 
$\,\gamma'\nh\in L\nh'\nh=L\cap\tvs\hn'$ and 
$\,\gamma''\nh\in L\nh''\nh=L\cap\tvs\hn''$ (see the lines preceding
(\ref{cnc})); in other words, $\,\gamma'\nh\circ\gamma''$ represents zero in 
$\,\hat L$.

Finally, $\,\hat L\,$ identified as above with a sub\-group of $\,\hat \bg\,$ 
is the kernel of the clear\-ly-sur\-jec\-tive homo\-mor\-phism 
$\,\hat\bg\to\hat H\nh$, induced by $\,\bg\to H\,$ in (\ref{exa}) (which, 
due to (e), proves (c)). Namely, $\,\hat L\,$ contains the kernel (the 
other inclusion being obvious): if the linear part of $\,\gamma\in\bg\,$ lies 
in $\,H\nh'\nnh H\nh''\nnh$, and so equals the linear part of 
$\,\gamma'\nh\circ\gamma''$ for some 
$\,(\gamma'\nnh,\gamma'')\in\bg\hn'\nh\times\bg\hn''\nnh$, then 
$\,\gamma=\lambda\circ\gamma'\nh\circ\gamma''\nnh$, where $\,\lambda\in L$.
\end{proof}

\section{Leaves and integral ho\-mol\-o\-gy}\label{ih}
This section once again employs the assumptions and notation of 
Theorem~\ref{restr}, with $\,\dim\tvs\nh=n\,$ and $\,\dim\tvs\hn'\nh=k$, where 
$\,0<k<n$. As the holonomy group 
$\,H\subseteq\mathrm{Iso}\hs\tvs\cong\mathrm{O}\hs(n)$ is finite 
(Remark~\ref{frpdc}), $\,\mathrm{det}\hh(H)\subseteq\{1,-\nh1\}$. 
In other words, the elements of $\,H\,$ have the determinants $\,\pm\nh1$. 
Using the covering projection 
$\,\mathcal{T}\nnh\to\mathcal{M}=\mathcal{T}\nnh/\nh H\nh$, cf.\ 
(\ref{cmp}) and the line following it, we see that
\begin{equation}\label{ori}
\mathrm{the\ condition\ }\,\mathrm{det}\hh(H)=\{1\}\,\mathrm{\ amounts\ to\
o\-ri\-ent\-a\-bil\-i\-ty\ of\ }\,\mathcal{M}\hh.
\end{equation}
By Theorem~\ref{gnric}(iii), the generic leaves 
of $\,F\hskip-3.8pt_{\mathcal{M}}\w$, defined as in the line following 
(\ref{krh}), are either all o\-ri\-ent\-a\-ble or all non\-o\-ri\-ent\-a\-ble.
\begin{theorem}\label{orien}Let\/ $\,\mathcal{M}\,$ be 
o\-ri\-ent\-a\-ble. Then all the generic leaves\/ $\,\mathcal{M}\hn'$ of\/ 
$\,F\hskip-3.8pt_{\mathcal{M}}\w\hs$ may be o\-ri\-ent\-ed so as to represent 
the same nonzero\/ $\,k$-di\-men\-sion\-al real ho\-mol\-o\-gy class\/ 
$\,[\mathcal{M}\hn']\in H\nnh_k\w(\mathcal{M},\bbR)$, 
while the first two cardinalities in Theorem\/~{\rm\ref{intrs}(c)} equal the
intersection numbers of the real ho\-mol\-o\-gy classes 
$\,[\mathcal{M}\hn'],[\mathcal{M}\hn'']$, or 
$\,[\mathcal{T}\hh'],[\mathcal{T}\hh'']$.
\end{theorem}
\begin{proof}A fixed orientation of $\,\tvs\hn'\nnh$, being preserved, due to 
(\ref{kpr}) -- (\ref{axb}) and (\ref{ori}), by the generic stabilizer group 
$\,\stb\hn'\nnh$, gives rise to orientations of all the leaves 
$\,\mathcal{T}\hh'$ of $\,F\hskip-3pt_{\mathcal{T}}\w$ and all the
generic leaves 
$\,\mathcal{M}\hn'$ of $\,F\hskip-3.8pt_{\mathcal{M}}\w$, so as to make 
the covering projections $\,\mathcal{T}\hh'\nnh\to\nh\mathcal{M}\hn'$ in the 
line following (\ref{cpr}) o\-ri\-en\-ta\-tion-pre\-serv\-ing. Since the torus 
group $\,\tvs\nnh/\nh L\,$ acts transitively on the o\-ri\-ent\-ed leaves 
$\,\mathcal{T}\hh'\nnh$, they all represent a single real ho\-mol\-o\-gy class 
$\,[\mathcal{T}\hh']\in H\nnh_k\w(\mathcal{T}\nh,\bbR)$, equal to the image of 
the fundamental class of $\,\mathcal{T}\hh'$ under the inclusion 
$\,\mathcal{T}\hh'\nnh\to\mathcal{T}\nnh$. At the same time, for generic 
leaves $\,\mathcal{M}\hn'\nnh$, the $\,d\hh$-fold covering projection 
$\,\mathcal{T}\hh'\nnh\to\nh\mathcal{M}\hn'$ (where $\,d=|H\nh'|$ does not 
depend on the choice of $\,\mathcal{M}\hn'\nnh$, cf.\ Lemma~\ref{subgp}) 
sends the fundamental class of $\,\mathcal{T}\hh'$ to $\,d\,$ times the 
fundamental class of $\,\mathcal{M}\hn'\nnh$. Thus, by functoriality, 
$\,d[\mathcal{M}\hn']\in H\nnh_k\w(\mathcal{M},\bbR)$ is the image of 
$\,[\mathcal{T}\hh']\in H\nnh_k\w(\mathcal{T}\nh,\bbR)$ under the covering 
projection $\,\mathcal{T}\nnh\to\mathcal{M}$, which makes it the same for all 
the generic leaves $\,\mathcal{M}\hn'\nnh$. Finally,
$\,[\mathcal{M}\hn']\ne0$, 
since a fixed constant positive differential $\,k$-form on the o\-ri\-ent\-ed 
space $\,\tvs\hn'$ descends, in view of the first line of this proof, to a 
parallel positive volume form on each o\-ri\-ent\-ed generic leaf 
$\,\mathcal{M}\hn'$ which yields a positive value when integrated over 
$\,[\mathcal{M}\hn']$.
\end{proof}
Note that the final clause in Theorem~\ref{orien} 
is, not surprisingly, consistent with the fact that -- by
Theorem~\ref{intrs}(a) 
and Lemma~\ref{subgp} -- the intersection numbers depend just on the two
mutually complementary $\,H\nnh$-in\-var\-i\-ant $\,L$-sub\-spaces
$\,\tvs\hn'\nnh,\tvs\hn''$ of 
$\,\tvs\nnh$, and not on the individual generic leaves 
$\,\mathcal{M}\hn'\nnh,\mathcal{M}\hn''\nnh,\mathcal{T}\hh'$ or 
$\,\mathcal{T}\hh''\nnh$.

\section{Generalized Klein bottles}\label{gk}
This section presents some known examples \cite[p.\ 163]{charlap} to 
illustrate our discussion.

The symbols $\,\bbR\nnh^H\nnh,\bbZ\nh^H$ used below follow the
set-the\-o\-ret\-i\-cal notational convention:
$\,Y\nh^X\hs$ is {\it the set of all mappings\/} $\,X\hn\to Y\nh$, not 
the fix\-ed-point set of some group action.

Let $\,S^1$ and $\,r\hskip-3pt_\theta\w:S^1\nh\to S^1$ denote the 
unit circle in $\,\bbC\,$ and the rotation by angle $\,\theta$ 
(multiplication by $\,e^{i\theta}$). For a fixed integer $\,n\ge2\,$ and 
the group $\,H\nh=\bbZ\nh_n\w\subseteq S^1$ of $\,n$th roots of unity, 
$\,\bbZ\nh^H\cong\bbZ^n$ is a lattice in the Euclidean space 
$\,\tvs\nh=\bbR\nnh^H\cong\rn$ with the $\,\ell\hh^2$ inner product, and
$\,\bbZ_0^{\hskip-.4ptH}
=\{\psi\in\bbZ\nh^H:\psi\nnh_{\mathrm{avg}}\w\nh=0\}\,$ is a sub\-group of 
$\,\bbZ\nh^H$ isomorphic to $\,\bbZ^\nmo\nnh$, where 
$\,(\hskip2.7pt)\nh_{\mathrm{avg}}\w$ denotes the averaging functional 
$\,\tvs\nh\to\bbR$. Setting 
$\,\bg=[(\hskip-1.9pt1\nh/\nh n\nh)\bbZ]\times\bbZ_0^{\hskip-.4ptH}$, one easily sees
that the assignment
\begin{equation}\label{tpf}
((t,\psi),f)\,\mapsto\,f\nh\circ r\nnh_{2\pi t}\w\nh+\,t\,+\,\psi\hh,\hskip12pt
\mathrm{where\ }\,\,(t,\psi)\in\bg\hs\mathrm{\ and\
}\,f\in\tvs\nh=\bbR\nnh^H\nnh,
\end{equation}
defines an \af\ isometric action on $\,\tvs\hs$ by $\,\bg\,$ treated as
a group with the group operation 
$\,(t,\psi)(t'\nh,\psi')
=(t\,+\,t',\,\psi'\nnh\circ r\nnh_{2\pi t}\w\nnh+\psi)$.
The term $\,t\,$ in (\ref{tpf}) is the constant function 
$\,t:H\nh\to\bbR$. Note that, in the right-hand side of (\ref{tpf}), as 
$\,(t,\psi)\in\bg\hh$,
\begin{equation}\label{avg}
t_{\mathrm{avg}}\w\hs=\,t\hh,\hskip28pt
\psi\nnh_{\mathrm{avg}}\w\hs=\,0\hh,\hskip28pt
(f\nnh\circ\nnh r\nnh_{2\pi t}\w)\nh_{\mathrm{avg}}\w\hs
=\,f\hskip-2.3pt_{\mathrm{avg}}\w\hh.
\end{equation}
\begin{proposition}\label{biebg}
These\/ $\,H\nh,\tvs\,$ and\/ $\,\bg\hs$ have the following properties.
\begin{enumerate}
  \def\theenumi{{\rm\roman{enumi}}}
\item The action of\/ $\,\bg\,$ on\/ $\,\tvs\,$ is effective and free.
\item $\bg\,$ is a Bie\-ber\-bach group in the underlying Euclidean 
\af\/ $\,n$-space of\/ $\,\tvs\nh$.
\item The holonomy group and lattice sub\-group of\/ $\,\bg\,$ are our\/ 
$\,H\cong\bbZ\nh_n\w$, acting on\/ $\,\tvs\hs$ linearly by\/ 
$\,H\times\tvs\ni(e^{i\theta}\nh,f)
\mapsto f\nh\circ r\hskip-3pt_\theta\w\in\tvs\nnh$, 
and\/ $\,L=\bbZ\times\nh\bbZ_0^{\hskip-.4ptH}\nnh$.
\item \,As a transformation of\/ $\,\tvs\nnh$, each\/ $\,(t,\psi)\in L\,$
equals the translation by\/ $\,t+\psi$.
\item $L\,$ consists of all translations by vectors\/ 
$\,\psi\hn'\in\bbZ\nh^H\,$ such that\/ 
$\,\psi\hn'_{\hskip-1.5pt\mathrm{avg}}\nh\in\bbZ\hh$.
\end{enumerate}
An example of two mutually complementary\/ $\,H\nh$-in\-var\-i\-ant\/ 
$\,L$-sub\-spaces of\/ $\,\tvs\nh$, in the sense of\/ {\rm(\ref{cpl})} and 
Definition\/~{\rm\ref{lsbsp}}, is provided by the line\/ $\,\tvs\hn'$ of 
constant functions $\,H\to\bbR\,$ and the hyper\-plane\/ $\,\tvs\hn''$ 
consisting of all\/ $\,f:H\to\bbR\,$ with\/ 
$\,f\hskip-2.3pt_{\mathrm{avg}}\w=0$. The generic stabilizer groups 
$\,\stb\hn'\nnh,\stb\hn''\nnh\subseteq\bg\,$ associated via\/ 
{\rm(\ref{krh})} with\/ $\,\tvs\hn'$ and\/ $\,\tvs\hn''$ are the translation 
groups\/ $\,\bbZ\times\nnh\{0\}\,$ and\/ 
$\,\{0\}\nnh\times\bbZ_0^{\hskip-.4ptH}\nnh$, both contained in\/ $\,L$.
Furthermore,
\begin{enumerate}
  \def\theenumi{{\rm\alph{enumi}}}
\item under the obvious identifications of\/ 
$\,\tvs\nnh/\hn\tvs\hn'$ with\/ $\,\tvs\hn''$ and\/ 
$\,\tvs\nnh/\hn\tvs\hn''$ with\/ $\,\tvs\hn'\nh$, the quotient actions of\/
$\,\bg\,$ become\/ 
$\,\bg\times\tvs\hn''\ni((t,\psi),f)
\mapsto f\nh\circ r\nnh_{2\pi t}\w\nnh+\psi\in\tvs\hn''$ and, respectively, 
$\,\bg\times\tvs\hn'\ni((t,\psi),f)
\mapsto f\nnh+t\in\tvs\hn'\nnh$,
\item every coset of the\/ $\,L$-sub\-space\/ $\,\tvs\hn''$ is generic, as
  defined in Sect.~{\rm\ref{gg}},
\item non\-ge\-ner\-ic cosets of\/ $\,\tvs\hn'$ are precisely those cosets
containing\/ $\hs f:H\to\bbR\hs$ such that\/ 
$\,f\circ r\nnh_{2\pi t}\w\nh-f\,$ is in\-te\-ger-val\-ued for some\/ 
$\,t\in[(\hskip-1.9pt1\nh/\nh n\nh)\bbZ]\smallsetminus\bbZ\hh$,
\item the obvious homomorphism\/
$\,\bg=[(\nh\nnh1\nh/\nh n\nh)\bbZ]\times\bbZ_0^{\hskip-.4ptH}
\to(\hskip-1.9pt1\nh/\nh n\nh)\bbZ\,$ maps the stabilizer group of each coset of\/ 
$\,\tvs\hn'$ iso\-mor\-phic\-al\-ly onto a sub\-group of\/ 
$\,(\hskip-1.9pt1\nh/\nh n\nh)\bbZ\hh$,
\item the sub\-groups of\/ $\,(\hskip-1.9pt1\nh/\nh n\nh)\bbZ\,$ resulting from\/
{\rm(c)} have the form\/ $\,(\hskip-1.4ptd\nh/\nh n\nh)\bbZ$, where\/ $\,d\,$ is a 
positive divisor of\/ $\,n\,$ or, equivalently, are the pre\-im\-ages,
under the homo\-mor\-phism\/
$\,\bbR\ni t\mapsto e^{2\pi i\hh t}\nnh\in S^1\nh$, of sub\-groups of the
group\/ $\,H\nh=\bbZ\nh_n\w\subseteq S^1$ formed by the\/ $\,n$th roots 
of unity.
\end{enumerate}
\end{proposition}
\begin{proof}First, $\,\bg\,$ acts on $\,\tvs\hs$ freely: if 
$\,f\nh\circ r\nnh_{2\pi t}\w\nh+t+\psi=f\nh$, cf.\ (\ref{tpf}), with 
$\,f:H\to\bbR$, applying $\,(\hskip2.7pt)\nh_{\mathrm{avg}}\w$ to both sides, 
we get, by (\ref{avg}), $\,t=0$, and hence 
$\,f\nnh\circ\nh r\nnh_{2\pi t}\w\nh=f\nh$, so that the equality 
$\,f\nnh\circ\nh r\nnh_{2\pi t}\w\nh+t+\psi=f\,$ reads $\,\psi=0$. Secondly, 
$\,H\,$ and $\,L\,$ {\it defined\/} by (iii) arise from $\,\bg\,$ as required
in (\ref{lah}): the claim about $\,H\,$ is obvious, and so are (iv) -- (v), 
showing that $\,L\subseteq\bg\hn\cap\nnh\tvs\nnh$. Conversely, 
$\,\bg\hn\cap\tvs\subseteq\hn L$. To verify this, suppose that 
$\,f\nh\circ r\nnh_{2\pi t}\w\nh+t+\psi=f\nh+\psi\hn'$ for all 
$\,f\in\tvs\nh=\bbR\nnh^H\nnh$, some 
$\,(t,\psi)\in\bg\hh$, and some $\,\psi\hn'\nh\in\tvs\nnh$. Taking the 
linear parts of both sides, we see that $\,t\in\bbZ\,$ and $\,(t,\psi)\in L$, 
as required.

Our $\,\bg\,$ has a compact fundamental domain in $\,\tvs\nnh$, since so does 
the lattice $\,L\subseteq\bg$. Also, $\,\bg\,$ is  
tor\-sion-free: $\,\bg\ni(t,\psi)\mapsto t\in\bbR\,$ being a group 
homo\-mor\-phism, any fi\-nite\hh-or\-der element $\,(t,\psi)\,$ of $\,\bg\,$ 
has $\,t=0$, and so, by (\ref{tpf}), it acts via translation by $\,\psi$, 
which gives $\,\psi=0$. Next, to establish the discreteness of the subset 
$\,\bg\,$ of $\,\mathrm{Iso}\hs\tvs\hs$ (and, consequently, (ii)), suppose 
that a sequence $\,(t\nh_k\w,\psi\hskip-2pt_k\w)\in\bg\,$ with pairwise
distinct terms yields, via (\ref{tpf}), a sequence convergent in 
$\,\mathrm{Iso}\hs\tvs\nnh$. Evaluating (\ref{tpf}) on $\,f\nh=0$, we get 
$\,(t\nh_k\w,\psi\hskip-2pt_k\w)\to(t,\psi)\,$ in $\,\bbR\times\bbR\nnh^H$ as 
$\,k\to\infty$, for some $\,(t,\psi)\,$ and, since 
$\,(t\nh_k\w,\psi\hskip-2pt_k\w)
\in[(\hskip-1.9pt1\nh/\nh n\nh)\bbZ]\times\bbZ_0^{\hskip-.4ptH}\nnh$, the sequence 
$\,(t\nh_k\w,\psi\hskip-2pt_k\w)\,$ becomes eventually constant, contrary to
its terms' being pairwise distinct.

As for $\,\tvs\hn'$ and $\,\tvs\hn''\nh$, note that, by (iv), a
$\,\bbZ\nh$-ba\-sis of 
$\,L\cap\tvs\hn'$ (or, $\,L\cap\tvs\hn''$) may be defined to consist just 
of the constant function $\,1\,$ (or, respectively, of the $\,\nmo$ 
functions $\,\psi\nnh_q:H\to\bbZ$, labeled by $\,q\in H\smallsetminus\{1\}$, 
where $\,\psi\nnh_q(q)=1=-\hh\psi\nnh_q(1)$ and $\hs\psi\nnh_q=0\hs$ on 
$\,H\smallsetminus\{1,q\}$). Specifically, 
$\,\psi\hn'\nh=\sum_q\w\psi\hn'\hn(q)\hs\psi\nnh_q$ whenever 
$\,\psi\hh'\in\bbZ\nh^H$ and
$\,\psi\hn'_{\hskip-1.5pt\mathrm{avg}}\nh=0$. Our descriptions of
$\,\stb\hn'$ and $\,\stb\hn''$ are in turn immediate from (a), which itself
is a trivial consequence of (\ref{tpf}) -- (\ref{avg}), and 
easily implies (b) -- (c). Now (d) follows as the relation
$\,f\nh\circ r\nnh_{2\pi t}\w\nnh+\psi=f\nh$, with fixed
$\,f\in\tvs\hn''\nnh$, uniquely determines $\,\psi$, once $\,t$ is given.
Finally, we have (e) since the first assignment in (a) only depends on $\,t$ 
through $\,e^{2\pi i\hh t}\nh$, which completes the proof.
\end{proof}
The compact flat Riemannian manifold $\,\tvs\nnh/\hh\bg\,$ arising here from
our Bie\-ber\-bach group $\,\bg\,$ as in Sect.~\ref{bg} is called the 
$\,n$-di\-men\-sion\-al {\it generalized Klein bottle\/} 
\cite[p.\ 163]{charlap}. The linear functional 
$\,\tvs\ni f\mapsto f\hskip-2.3pt_{\mathrm{avg}}\w\in\bbR\,$ is 
equi\-var\-i\-ant, due to (\ref{avg}), with respect to the actions of 
$\,\bg\,$ and $\,\bbZ$ (the latter, on $\,\bbR$, via translations by 
multiples of $\,\nh1\nh/\nh n$), relative to the homo\-mor\-phism 
$\,\bg\ni(t,\psi)\mapsto t\in(\hskip-1.9pt1\nh/\nh n\nh)\bbZ\hh$. Thus, it descends, in 
view of Remark~\ref{covpr}(c), to a bundle projection 
$\,\tvs\nnh/\hh\bg\to\bbR/[(\hskip-1.9pt1\nh/\nh n\nh)\bbZ]$, making 
$\,\tvs\nnh/\hh\bg\,$ a bundle of tori over the circle. The fibres of this 
bundle are, obviously, the images, under the projection
$\,\mathrm{pr}:\tvs\nh\to\tvs\nnh/\hh\bg\hh$, of cosets of the
$\,L$-sub\-space 
$\,\tvs\hn''\nh\subseteq\tvs\hs$ mentioned in Proposition~\ref{biebg}(b), all
of them generic. On the other hand, $\,\tvs\hn'$ has 
some non\-ge\-ner\-ic cosets -- by Proposition~\ref{biebg}(c), an example is
$\,\tvs\hn'$ itself, with the stabilizer group easily seen to be
$\,[(\hskip-1.9pt1\nh/\nh n\nh)\bbZ]\times\nnh\{0\}$. The
$\,\mathrm{pr}$-im\-ages of
the cosets of $\,\tvs\hn'$ are embedded circles, forming the leaves
of the foliation $\,F\hskip-3.8pt_{\mathcal{M}}\w$ of
$\,\mathcal{M}=\tvs\nnh/\hh\bg\,$ arising as in Theorem~\ref{restr}. 

For the foliation $\,F\hskip-3.8pt_{\mathcal{M}}\w$ of 
$\,\mathcal{M}=\tvs\nnh/\hh\bg\,$ obtained in the general case of
Theorem~\ref{restr}, the stabilizer group of a leaf $\,\mathcal{M}\hn'$ of 
$\,F\hskip-3.8pt_{\mathcal{M}}\w$ is only defined as a conjugacy class of 
sub\-groups of $\,\bg\,$ (the sub\-groups being the stabilizer groups
$\,\stb\hn'$ of leaves $\,\mathcal{E}\nh'$ of $\,F\hskip-3pt_{\mathcal{E}}\w$
with $\,\mathrm{pr}(\mathcal{E}\nh')=\mathcal{M}\hn'$). Genericity of 
$\,\mathcal{M}\hn'\nh$, mentioned in the line following (\ref{krh}), amounts
to genericity of all such $\,\mathcal{E}\nh'\nh$, and $\,\stb\hn'$ is then
uniquely associated with $\,\mathcal{M}\hn'$ (due to its being the generic
stabilizer group, normal in $\,\bg$). In the subsequent discussion
$\,\stb\hn'$ 
is also treated as uniquely defined, for a different reason: each
$\,\stb\hn'$ is replaced by its image under a homo\-mor\-phism from
$\,\bg\,$ into the
Abel\-i\-an group $\,(\hskip-1.9pt1\nh/\nh n\nh)\bbZ$.

We denote by $\,|\,$ the divisibility relation in the set $\,\Delta_n\w$ of
all positive divisors of $\,n\,$ and by $\,\mathcal{M}\,$ 
the $\,n$-di\-men\-sion\-al generalized Klein bottle $\,\tvs\nnh/\hh\bg\hh$. 
Let us also use Proposition~\ref{biebg}(d)\hs-\hs(e) and the final sentence of
the last paragraph to treat
\begin{equation}\label{tsg}
\begin{array}{l}
\mathrm{the\ stabilizer\ groups\ of\ leaves\ of\
}\,F\hskip-3.8pt_{\mathcal{M}}\w\hs\mathrm{\ as\ sub\-groups}\\
\mathrm{of\ }\,(\hskip-1.9pt1\nh/\nh n\nh)\bbZ\,\mathrm{\ having\
the\ form\ }\,(\hskip-1.4ptd\nh/\nh n\nh)\bbZ\hh\mathrm{,\ where\
}\,d\in\Delta_n\w\hn.
\end{array}
\end{equation}
\begin{proposition}\label{nongn}
There exists a family\/ $\,\{\mathcal{M}[d]:d\in\Delta_n\w\}\,$ of
compact connected immersed sub\-man\-i\-folds of\/ $\,\mathcal{M}\,$ with the
following properties, for all\/ $\,d,d'\in\Delta_n\w$.
\begin{enumerate}
  \def\theenumi{{\rm\roman{enumi}}}
\item Each\/ $\,\mathcal{M}[d]\,$ has the dimension\/ $\,d\,$ and is foliated
by circle leaves of\/ $\,F\hskip-3.8pt_{\mathcal{M}}\w$.
\item $\mathcal{M}[d']\subseteq\mathcal{M}[d]\,$ whenever\/ $\,d'\hn|\hs d$.
\item $\mathcal{M}[d]\smallsetminus\bigcup_k\w\nnh\mathcal{M}[k]$, with\/
$\,k\,$ ranging over $\,\Delta\nnh_d\w\smallsetminus\{d\}$, equals the union 
of all leaves of\/ $\,F\hskip-3.8pt_{\mathcal{M}}\w$ having the stabilizer 
group\/ $\,(\hskip-1.4ptd\nh/\nh n\nh)\bbZ\hh$.
\item $\mathcal{M}[n]=\mathcal{M}\,$ and\/
$\,\mathcal{M}[1]=\mathrm{pr}(\tvs\hn')$.
\item If\/ $\,n\,$ is prime, $\,\mathcal{M}[1]=\mathrm{pr}(\tvs\hn')\,$ is the
only non\-ge\-ner\-ic leaf of\/ $\,F\hskip-3.8pt_{\mathcal{M}}\w$.
\end{enumerate}
\end{proposition}
\begin{proof}The ze\-ro-av\-er\-age functions $\,h:H\to\bbR/\bbZ\hh$, from the 
group $\,H\nh=\bbZ\nh_n\w\subseteq S^1$ of $\,n$th roots of unity into the
circle $\,\bbR/\bbZ\hh$, form a manifold $\,[\bbR/\bbZ]\nh^H$
dif\-feo\-mor\-phic
to the torus $\,T^{n-1}\nnh$, and the covering projection 
$\,\tvs\hn''\nh\to[\bbR/\bbZ]\nh^H\nnh$, sending a ze\-ro-av\-er\-age function
$\,f:H\to\bbR\,$ to its composition $\,h\,$ with the projection
$\,\bbR\to\bbR/\bbZ\hh$, is obviously equi\-var\-i\-ant relative to the first
action in Proposition~\ref{biebg}(a),
\begin{equation}\label{ach}
\mathrm{the\ action\ of\ }\,H\,\mathrm{\ on\ }\,[\bbR/\bbZ]\nh^H\mathrm{\
given\ by\ }\,(q,h)\mapsto h\circ r\nnh_{2\pi t}\w,
\end{equation}
for $\,t\in(\hskip-1.9pt1\nh/\nh n\nh)\bbZ\,$ with $\,q=e^{2\pi i\hh t}\nh$,
and the homo\-mor\-phism $\,\bg\ni(t,\psi)\mapsto e^{2\pi i\hh t}\in H$.
Thus, (\ref{ach}) is the translational action of $\,H\nh=\bbZ\nh_n\w$ on
mappings $\,H\to\bbR/\bbZ$, and the pre\-im\-ages of its iso\-tropy groups
under the homo\-mor\-phism
$\,(\hskip-1.9pt1\nh/\nh n\nh)\bbZ\ni t\mapsto e^{2\pi i\hh t}\in H$ are
precisely the stabilizer groups in (\ref{tsg}). We now define
$\,\mathcal{M}[d]$, for any $\,d\in\Delta_n\w$, to be the union of
circle leaves of $\,F\hskip-3.8pt_{\mathcal{M}}\w$ having stabilizer groups
contained in $\,\,(\hskip-1.4ptd\nh/\nh n\nh)\bbZ\hh$.
Thus, $\,\mathcal{M}[d]\,$ is the image under the quotient projection
$\,\mathrm{pr}:\tvs\nh\to\tvs\nnh/\hh\bg\,$ of the union of all cosets
$\,v\hs+\nh\tvs\hn'$ for vectors $\,v\in\tvs\hn''$ which the covering
projection 
$\,\tvs\hn''\nh\to[\bbR/\bbZ]\nh^H$ sends to (ze\-ro-av\-er\-age) functions
$\,h:H\to\bbR/\bbZ\,$ having $\,h\circ r\nnh_{2\pi d/n}\w=h$. Since such
$\,h\,$ form a manifold dif\-feo\-mor\-phic  
to the torus $\,T^{d-1}\nnh$, via the assignment to $\,h\,$ of the ze\-ro-sum
$\,d$-tuple consisting of $\,h(e^{2\pi i\hh k/n})\in\bbR/\bbZ$, with
$\,k=0,1,\dots,d-1$, all our claims about $\,\mathcal{M}[d]\,$ easily follow
from
the fact that both projections just mentioned are locally dif\-feo\-mor\-phic.
\end{proof}
The $\,n$-di\-men\-sion\-al generalized Klein bottle, for any $\,n\ge2$, is
an example illustrating the fact 
that the last inclusion of Theorem~\ref{restr}(ii-c) may be proper. 
In fact, the stabilizer group $\,[(\hskip-1.9pt1\nh/\nh n\nh)\bbZ]\times\nnh\{0\}\,$ of
$\,\tvs\hn'\nnh$, mentioned six lines before 
Proposition~\ref{nongn}, although not contained in the lattice 
$\,L$, acts on $\,\tvs\hn'$ by translations. Also, unless $\,n\,$ is prime, 
Proposition~\ref{nongn}(iii) shows that the dependence on $\,u\,$ in the last
line of Lemma~\ref{isolv} is actually possible, as it occurs for
$\,\mathcal{M}\hn'\nnh=\mathcal{M}[k]$, with any
$\,k\in\Delta_n\w\smallsetminus\{n\}$.

\section{Remarks on holonomy}\label{rh}
The correspondence between Bie\-ber\-bach groups and compact flat manifolds
mentioned in Remark~\ref{bijct} has an extension to al\-most-Bie\-ber\-bach
groups and in\-fra-nil\-man\-i\-folds \cite{dekimpe} obtained by using
-- instead of the 
translation vector space of a Euclidean \af\ space -- a connected, simply 
connected nil\-po\-tent Lie group $\,\sg\,$ acting simply transitively 
on a manifold $\,\mathcal{E}$, and replacing the Bie\-ber\-bach group with 
a tor\-sion-free uniform discrete sub\-group 
$\,\bg\,$ of $\,\mathrm{Dif{}f}\,\mathcal{E}$ contained in a 
sem\-i\-di\-rect product (canonically transplanted so as to act on 
$\,\mathcal{E}$) of $\,\sg\,$ and a maximal compact sub\-group of 
$\,\mathrm{Aut}\,\sg\nh$. Here `uniform' means admitting a compact fundamental 
domain, cf.\ Remark~\ref{cptfd}. The analogs of (\ref{exa}) and (\ref{cmp}) 
remain valid, reflecting the fact that any in\-fra-nil\-man\-i\-fold is the 
quotient of a nil\-man\-i\-fold under the action of a finite group $\,H\nnh$.

A somewhat similar picture may arise in some cases where $\,\sg\,$ is not 
assumed nil\-po\-tent. As an example, let  
$\,\sg\cong\mathrm{Spin}\hs(m,1)\,$ be the universal covering group of the 
identity component $\,\sg/\bbZ_2\w\cong\mathrm{SO}^+\nh(m,1)\,$ of the 
pseu\-\hbox{do\hs-}\hskip0ptor\-thog\-onal group of an 
$\,(m+1)$-di\-men\-sion\-al Lo\-rentz\-i\-an vector space $\,\mathcal {L}$, 
$\,m\ge3$. Here $\,\mathcal{E}\hs$ is the \hbox{(two\hs-} fold) universal
covering 
manifold of the or\-tho\-nor\-mal-frame bundle of the future unit 
pseu\-do\-sphere $\,\mathcal{S}\subseteq\hn\mathcal{L}$, isometric to the 
hyperbolic $\,m$-space, and $\,\sg/\bbZ_2\w$ acts on $\,\mathcal{S}$ via 
hyperbolic isometries, leading to an action of $\,\sg\,$ on 
$\,\mathcal{E}$. The or\-tho\-nor\-mal-frame bundles of compact hyperbolic 
manifolds obtained as quotients of $\,\mathcal{S}\,$ give rise to the 
required tor\-sion-free uniform discrete sub\-groups $\,\bg$.

The resulting compact quotient manifolds $\,\mathcal{M}=\mathcal{E}/\hh\bg\,$ 
can be endowed with various interesting Riemannian metrics coming from 
$\,\bg$-in\-var\-i\-ant metrics on $\,\mathcal{E}$. For $\,\bg\,$ and 
$\,\mathcal{E}\hs$ of the preceding paragraph, a particularly natural choice 
of an invariant {\it indefinite\/} metric is provided by the Kil\-ling 
form of $\,\sg\nh$, turning $\,\mathcal{M}\,$ into a compact locally symmetric 
pseu\-\hbox{do\hskip1pt-}\hskip-.7ptRiem\-ann\-i\-an Ein\-stein manifold.

Outside of the Bie\-ber\-bach-group case, however, these metrics are not flat, 
and finite groups $\,H\,$ such as mentioned above cannot serve as their 
holonomy groups. The holonomy interpretation of $\,H\,$ still makes sense, 
though, if -- instead of metrics -- one uses either of the two
$\,\bg$-in\-var\-i\-ant flat connections, with (parallel) torsion,
naturally distinguished on $\,\mathcal{E}$. Here
$\,\mathcal{E}\hs$ is, again, a manifold on which a connected, simply 
connected Lie group $\,\sg\,$ acts simply transitively. Two natural 
bi-in\-var\-i\-ant connections with the stated properties exist on
$\,\sg$, rather than $\,\mathcal{E}$, and are characterized by the
requirement that they make all the left-in\-var\-i\-ant (or,
right-in\-var\-i\-ant) vector fields parallel. Due to their naturality,
these two connections on $\,\sg$ are also invariant under all the Lie-group
auto\-mor\-phisms of $\,\sg\nh$. It is the two connections on $\,\sg$ that
induce, in an obvious way, the ones on $\,\mathcal{E}$, mentioned six lines
earlier.

\setcounter{section}{1}
\renewcommand{\thesection}{\Alph{section}}
\setcounter{theorem}{0}
\renewcommand{\thetheorem}{\thesection.\arabic{theorem}}
\section*{Appendix: Hiss and Szcze\-pa\'n\-ski's reducibility
theorem}\label{hs}
Consider an {\it abstract Bie\-ber\-bach group}, that is, any 
tor\-sion-free group $\,\bg\,$ containing a finitely generated free 
A\-bel\-i\-an normal sub\-group $\,L\,$ of finite index, which is at the 
same time a maximal A\-bel\-i\-an sub\-group of $\,\bg$. As shown by 
Zas\-sen\-haus \cite{zassenhaus}, up to isomorphisms these groups coincide 
with the Bie\-ber\-bach groups of Sect.~\ref{bg}, we again 
summarize their structure using the short exact sequence
\begin{equation}\label{abs}
L\,\to\,\hs\bg\,\to\,H,\hskip12pt\mathrm{where\ }\,\,H\,=\,\bg/\nh L\hh.
\end{equation}
For A\-bel\-i\-an groups $\,G\hskip-1.9pt_1\w,G\hskip-1.5pt_2\w$ and $\,G'$ one 
has canonical isomorphisms
\begin{equation}\label{can}
\bbZ\otimes G\cong G\hh,\hskip7pt(G\hskip-1.9pt_1\w\nnh\oplus G\hskip-1.5pt_2\w)\otimes G'
\cong(G\hskip-1.9pt_1\w\nnh\otimes G')\oplus(G\hskip-1.5pt_2\w\nnh\otimes G')\hh,\hskip7pt
L\otimes\hs\bbQ\cong\mathrm{Hom}\hs(L\nh^*\nnh\nh,\bbQ)\hh,
\end{equation}
where $\,L\nh^*\nh=\mathrm{Hom}\hs(L,\bbZ)\,$ and, for simplicity, $\,L\,$ is 
assumed to be finitely generated and free. In the last case, with a suitable 
integer $\,n\ge0$, there are further noncanonical isomorphisms
\begin{equation}\label{ncn}
\mathrm{a)}\hskip6ptL\cong\bbZ^n,\hskip22pt\mathrm{b)}\hskip6ptL\otimes\hs\bbQ
\cong\bbQ^n\nh,
\end{equation}
while, using the injective homo\-mor\-phism 
$\,L\ni\lambda\mapsto\lambda\otimes\nh1\in L\hn\otimes\hh\bbQ\,$ to treat 
$\,L\,$ as a sub\-group of $\,L\otimes\nh\bbQ$, we see that, under suitably 
chosen identifications (\ref{ncn}),
\begin{equation}\label{crs}
\mathrm{the\ inclusion\ }\,L\subseteq L\otimes\hh\bbQ\,\mathrm{\ corresponds\ 
to\ the\ standard\ inclusion\ }\,\bbZ^n\nh\subseteq\hs\bbQ^n\nh.
\end{equation}
Finally, if $\,L\,$ as above is a (full) lattice in an 
fi\-nite-di\-men\-sion\-al real vector space $\,\tvs\nh$, a further canonical
isomorphic identification arises:
\begin{equation}\label{spn}
L\otimes\hh\bbQ\cong\mathrm{Span}_\bbQ\w L\hh,
\end{equation}
that is, we may view $\,L\otimes\hh\bbQ\,$ as the rational vector subspace of 
$\,\tvs$ spanned by $\,L$.

Let $\,\bg\,$ now be an abstract Bie\-ber\-bach group. Hiss and
Szcze\-pa\'n\-ski 
\cite[the corollary in Sect.\ 1]{hiss-szczepanski} proved that, if $\,L\,$ 
in (\ref{abs}) satisfies (\ref{ncn}.a) with $\,n\ge2$, then {\it the} 
(obviously $\,\bbQ$-lin\-e\-ar) {\it action of\/} $\,H\,$ {\it on\/} 
$\,L\otimes\hs\bbQ\,$ {\it must be reducible}, in the sense of admitting 
a nonzero proper invariant rational vector sub\-space $\,\mathcal{W}\nh$.

Next, using (\ref{crs}), we may write 
$\,L=\bbZ^n\nh\subseteq\hs\bbQ^n\nh=L\otimes\hh\bbQ$, so that 
$\,\mathcal{W}\subseteq\hs\bbQ^n\subseteq\rn\nh$, and the closure 
$\,\tvs\hn'$ of $\,\mathcal{W}\hs$ in $\,\rn$ has the real dimension 
$\,\dim_\bbQ\w\hskip-3pt\mathcal{W}\hs$ (any $\,\bbQ$-ba\-sis of 
$\,\mathcal{W}\hs$ being, obviously, an $\,\bbR$-ba\-sis of $\,\tvs\hn'$). By 
clearing denominators, one can replace such a $\,\bbQ$-ba\-sis with one 
consisting of vectors in $\,L=\bbZ^n\nh$, and so, by Lemma~\ref{lttce}(a), 
the intersection $\,L\nh'\nh=L\cap\mathcal{W}\nh=L\cap\tvs\hn'$ is a lattice 
in $\,\tvs\hn'\nnh$. We thus obtain (\ref{qdr}).

A stronger version of Hiss and Szcze\-pa\'n\-ski's reducibility theorem,
established more recently by Lu\-tow\-ski \cite{lutowski}, states that 
the rational holonomy representations of any compact flat manifold other than 
a torus has at least two non\-e\-quiv\-a\-lent irreducible 
sub\-rep\-re\-sen\-ta\-tions.

\begin{acknowledgement}
Both authors' research was supported in part by a FAPESP\nh-\hs OSU 2015 
Regular Research Award (FAPESP grant: 2015/50265-6). The authors wish to thank 
Andrzej Szcze\-pa\'n\-ski for helpful comments. We also greatly appreciate
suggestions made by anonymous referees of earlier versions of the
manuscript, which allowed us to make the exposition much easier to follow. 
\end{acknowledgement}
\ethics{Competing Interests}{The authors have no conflicts of interest to
declare that are relevant to the content of this chapter.}

\end{document}